\documentclass[11pt]{article}
\usepackage{graphicx}
\usepackage{amsmath}
\usepackage{amsfonts}
\usepackage{subcaption}
\usepackage{footmisc}
\usepackage{algorithm}
\usepackage[noend]{algpseudocode}
\algrenewcommand\algorithmicdo{}
\algrenewcommand\algorithmicthen{}
\usepackage{booktabs}
\usepackage{todonotes}
\usepackage{bm}
\usepackage{xspace}
\usepackage{xcolor}
\usepackage[normalem]{ulem}
\usepackage{stmaryrd}
\usepackage{thm-restate}
\usepackage{pifont}
\usepackage{wrapfig}
\usepackage{oubraces}
\usepackage{nicefrac}
\allowdisplaybreaks

\renewcommand{\phi}{\varphi}

\usepackage{qtree}

\usepackage[normalem]{ulem}

\usepackage{microtype}

\usepackage[margin=1in]{geometry}
\usepackage{amsmath,amssymb,amsthm,mathtools,microtype}
\usepackage{needspace}
\usepackage[hidelinks]{hyperref}

\newtheorem{theorem}{Theorem}

\newtheorem{lemma}[theorem]{Lemma}

\DeclareMathOperator{\Span}{span}

\newcommand{\defeq}{\stackrel{\text{def}}{=}}
\newcommand{\R}{\mathbb R} %
\newcommand{\Rp}{{\mathbb R}_{+}} %

\newcommand{\calP}{\mathcal P}

\newcommand{\eat}[1]{}

\renewcommand{\epsilon}{\varepsilon}

\newcommand{\set}[1]{\{#1\}}                    %
\newcommand{\linw}{\mathit{linw}  }     %
\newcommand{\entw}{\mathit{entw}}       %
\newcommand{\adw}{\mathit{adw}}         %

\newcommand{\ceil}[1]{\lceil#1\rceil}

\usetikzlibrary{fit}
\usetikzlibrary{calc}
\usetikzlibrary{trees}

\newcommand{\subw}{\operatorname{subw}}
\newcommand{\normw}{\operatorname{normw}}

\newcommand{\TD}{\mathrm{TD}}
 \usepackage{authblk}

\renewcommand{\adw}{\operatorname{adw}}
\renewcommand{\linw}{\operatorname{linw}}
\renewcommand{\entw}{\operatorname{entw}}
\newcommand{\hwd}{\operatorname{w}}
\newcommand{\ZY}{\operatorname{ZY}}
\newcommand{\Ing}{\operatorname{Ing}}
\newcommand{\xs}[1]{X_{#1}}   %

\usepackage{tikz}
\usetikzlibrary{positioning}
\definecolor{grpI}{HTML}{28618D}
\definecolor{grpII}{HTML}{B3532B}
\definecolor{grpIII}{HTML}{7453A3}
\definecolor{grpIV}{HTML}{287F71}
\definecolor{grpV}{HTML}{A63766}
\definecolor{grpN}{HTML}{5A5A5A}   %
\tikzset{
  mainbag/.style={draw=black!55,rounded corners=3pt,align=left,
                  inner xsep=6pt,inner ysep=4pt,font=\small,
                  fill=black!3,text width=62mm},
  maxbag/.style={mainbag,draw=black!80,line width=.9pt,fill=black!8},
  leafbag/.style={draw=#1!85!black,rounded corners=2pt,align=center,
                  inner sep=2pt,font=\scriptsize,fill=#1!8,
                  text=#1!75!black,minimum width=18mm},
  leafbag/.default=grpN,
  tlink/.style={draw=black!55,line width=.7pt},
  bnote/.style={font=\scriptsize,color=black!70,align=left}
}

\title{Separating Notions of Graph Width: the Adaptive, Normal, Linear,
  Entropic, and Submodular Width}
\author[1]{Matthias Lanzinger}
\author[2,1]{Timo Camillo Merkl}
\author[3]{Dan Suciu}
\affil[1]{TU Wien, Austria}
\affil[2]{Max Planck Institute for Software Systems, Germany}
\affil[3]{University of Washington, United States}

\date{}

\begin{document}
\maketitle
\vspace{-1.2em}

\begin{abstract}
  We describe one explicit simple graph $G$ on $32$ vertices whose
  adaptive, normal, linear, entropic, and submodular widths are pairwise
  distinct.  We compute all these widths exactly, except for the
  entropic width, where we only give a lower and upper bound.  We use
  Ingleton's inequality~\cite{ingleton1971representation} and the
  Zhang--Yeung inequality~\cite{zhang1998characterization} for upper
  bounds, and give explicit constructions of modular, normal, linear,
  entropic, and non-entropic polymatroids for lower bounds.
\end{abstract}

\section{Introduction}

This work is motivated by the decision problem for graph pattern
matching: given a pattern graph $G$ and database graph $D$, check if
$D$ contains a homomorphic image of $G$.  Since the introduction of
the submodular width of a graph by
Marx~\cite{DBLP:journals/jacm/Marx13}, variations of this notion have
been used to prove both
upper~\cite{DBLP:journals/theoretics/KhamisNS25} and lower
bounds~\cite{DBLP:journals/pacmmod/FanKZ24,DBLP:conf/icalp/FanKZ23} on
the complexity of graph pattern matching.  For example, a pattern $G$
with submodular width $\subw$ can be checked in time
$\tilde O(N^{\subw})$, where $N$ is the size of the database graph.

However, there are several variations of the ``width'' notion that one
could consider.  The original definition of $\subw$ by Marx used a
\emph{polymatroid}, which is a monotone and submodular set-function on
the vertices of the pattern graph.  Bounds obtained from polymatroids
can be converted into an algorithm.  On the other hand, circuit lower
bounds require the more stringent notion of \emph{entropic
  width}~\cite{DBLP:journals/pacmmod/FanKZ24}, while fine-grained
complexity lower
bounds~\cite{DBLP:conf/icalp/FanKZ23,DBLP:conf/stoc/BringmannG25}
require the much stricter notion of \emph{clique embedding}, which in
turn is a restriction of the \emph{normal width}.

The question we address here is whether these notions of width
coincide. A separation of the ``polymatroid bound'' and the
``entropic bound'' was proved in~\cite{DBLP:conf/pods/Khamis0S17}
for bounds on the number of matches of a hypergraph pattern under
cardinality constraints and functional dependencies. These bounds
optimize the value of a polymatroid on the full set of variables.
The widths studied here additionally involve minimizing the maximum
bag value over tree decompositions, so a separation of those bounds
does not directly imply a separation of widths. Moreover, we seek
a separation for graph patterns under edge domination alone,
without functional dependencies and cardinality constraints.

Thus, it was an open question whether these variations of width
coincide for graphs.  When $G$ is a cycle of length $k$, then its
submodular width is $2-1/\ceil{k/2}$, and this width is attained by
using a \emph{normal polymatroid}.  Thus, for graphs $G$ that are a
simple cycle, the \emph{normal width} coincides with the
\emph{entropic width}, and with the \emph{submodular width}.  This
naturally raises the question whether these notions are equal for all
graphs $G$.

However, the ``boat query''~\cite{DBLP:conf/icalp/FanKZ23}, which is a
graph pattern with 8 nodes originally studied
in~\cite{DBLP:conf/icdt/JoglekarR16}, has a submodular width of 2,
which is strictly larger than the normal width.  Thus, equality of
normal- and submodular width fails for the boat query.  However, for
this query, the polymatroid witnessing $\subw=2$ is a \emph{linear
  polymatroid}, meaning that it can be realized by the rank function
in a vector space (definitions in Section~\ref{sec:basic}).  This, in
turn, raised the natural question whether the submodular width
coincides with the linear width, for every graph $G$.

We answer this question negatively.  We describe a single graph $G$
with 32 vertices whose adaptive, normal, linear, entropic, and submodular
widths are pairwise distinct.  We compute all these widths exactly,
except for the entropic width, for which we only provide a lower and
upper bound (which are sufficient to separate it from the other widths).

Our results were obtained with significant assistance from AI Agents:
we used both OpenAI's ChatGPT and Anthropic's Claude (different
authors used different agents).  We did not obtain these results in a
single shot, but instead went through a long journey until we arrived
at the current results.  We have entirely rewritten the AI-generated
output, making it human-readable, and highlighted the important
lessons carried by this construction.  In particular, we identify
clearly the role of Ingleton's
inequality~\cite{ingleton1971representation}, and of the Zhang--Yeung
inequality~\cite{zhang1998characterization} in these separation
results, and describe a general class of non-entropic polymatroids
(denoted $f_G$ in Section~\ref{sec:three:special}), all of which are
of independent interest.  By necessity, some tedious calculations
remain, but they are clearly isolated in the paper, are verified by
the \texttt{verify.py} script accompanying this paper, and their role
in the big picture is clearly explained.

\paragraph{Structure of the Paper.}
Section~\ref{sec:basic} introduces the classes of polymatroids.
Section~\ref{sec:widths} defines the corresponding graph widths and
states our main result.  Section~\ref{sec:inequalities} presents the
inequalities used for the upper bounds, and
Section~\ref{sec:three:special} gives examples that prepare the
lower-bound constructions.  Section~\ref{sec:graph} introduces the
32-vertex graph that witnesses the separation of the width notions.
It consists of a central graph on five vertices and scaffolding that
connects the bounds on the central graph to the widths of the full
graph.  Section~\ref{sec:core-bounds} establishes these central bounds,
and Section~\ref{sec:big-graph} uses them to prove the width separations.
Section~\ref{sec:conclusions} concludes with open problems.

\section{Classes of Polymatroids}
\label{sec:basic}

Fix a set of $n$ variables $\bm X = \set{X_1, \ldots, X_n}$.  A
\emph{polymatroid} is a function $h : 2^{\bm X} \rightarrow \Rp$ that
satisfies for all $U,V\subseteq \bm X$:
\begin{align}
  h(\emptyset) = & 0 \label{eq:normalized} \\
  h(U \cup V) \geq & h(U) \label{eq:monotone} \\
  h(U) + h(V) \geq & h(U\cup V) + h(U \cap V) \label{eq:submodular}
\end{align}
The second and third conditions are called \emph{monotonicity} and
\emph{submodularity}.  We write $\Gamma_n$ for the set of polymatroids
over $n$ variables, or, alternatively, $\Gamma_{\bm X}$ to emphasize
that the set of variables is $\bm X$.  This set is a cone: if $h$ is a
polymatroid and $c\geq 0$, then $c\cdot h$ is also a polymatroid.  We
consider several subsets of $\Gamma_n$:

\begin{description}
\item[Modular Polymatroids.] A polymatroid $h$ is \emph{modular} if
  equality holds in \eqref{eq:submodular} for all $U,V\subseteq\bm X$.
  Equivalently, there exist weights $w_i\geq0$ such that
  \[
    h(U)=\sum_{X_i\in U}w_i
    \qquad\text{for every }U\subseteq\bm X.
  \]
  The weights are the singleton values $w_i=h(X_i)$.  We write $M_n$
  for the set of modular polymatroids.

\item[Normal Polymatroids.] If $K \subseteq \bm X$, then we call the
  indicator function $h^K(U) := \bm 1_{U\cap K \neq \emptyset}$ a
  \emph{step function}.  Every step function is a polymatroid.  A
  \emph{normal polymatroid} is a nonnegative linear combination of step
  functions.  We write $N_n$ for the set of normal polymatroids.

  Normal polymatroids have the following equivalent description, up to
  closure.  Fix a set $\Omega$.  A \emph{coverage function} is
  $\varphi : \bm X \rightarrow \calP_{\text{fin}}(\Omega)$.  If
  $U \subseteq \bm X$ then we write $\varphi(U)$ for the union
  $\bigcup_{X_i \in U}\varphi(X_i)$.  Every function of the form
  $h(U) := c\cdot |\varphi(U)|$ with $c \geq 0$ is a normal
  polymatroid: it equals
  $\sum_{\emptyset \ne K \subseteq \bm X}c\,m_K\,h^K$, where $m_K$ is
  the number of $\omega \in \Omega$ with
  $\set{X_i \mid \omega \in \varphi(X_i)}=K$.  Conversely, a normal
  polymatroid $\sum_K m_K h^K$ is of this form exactly when the
  coefficients $m_K$ are commensurable (i.e.\ their ratio is a rational
  number), and in general it is a limit of such functions.  Since
  $\hwd_h(G)$ is continuous and positively homogeneous in $h$, the two
  descriptions give the same value of $\normw(G)$, and we use them
  interchangeably.  Every normal witness in this paper has rational
  coefficients, so it is literally of the form $c\cdot|\varphi(\cdot)|$.

\item[Linear Polymatroids.] Consider a vector space $\R^d$ and let
  $\varphi : \bm X \rightarrow \calP_{\text{fin}}(\R^d)$ be a coverage
  function.  The function $r : 2^{\bm X} \rightarrow \Rp$ defined by:
  \begin{align*}
    r(U) := & \dim(\Span(\varphi(U)))
  \end{align*}
  is called a \emph{rank function}.  One can check that $r$ is a
  polymatroid.  A polymatroid is called \emph{linear} if it is a limit
  of nonnegative scalar multiples of rank functions.  We write
  $\text{Lin}_n$ for the set of linear polymatroids.

\item[Entropic and Almost Entropic Polymatroids.] Let $(X_i)_{i=1}^{n}$
  be a set of joint random variables, and set $h(U) := H(U)$ the
  entropy of the subset $U \subseteq \bm X$.  Then $h$ is a
  polymatroid, called an \emph{entropic polymatroid}.  The set of
  entropic polymatroids is denoted $\Gamma_n^*$.  Its topological
  closure is denoted $\bar \Gamma_n^*$, and its elements are called
  \emph{almost entropic functions}.
\end{description}

It is known that
$
M_n \subseteq N_n \subseteq \text{Lin}_n \subseteq \bar\Gamma^*_n
\subseteq \Gamma_n,
$
with all inclusions strict for $n\ge4$.

\section{Graph Widths and the Separation Theorem}
\label{sec:widths}

All graphs are finite, simple, and have no isolated vertices.
A \emph{tree decomposition} of a
graph $G=(V,E)$ is a tree $T$ together with bags $B_t\subseteq V$
($t\in V(T)$) such that every edge of $G$ is contained in some bag and,
for every $x\in V$, the set $\{t \mid x\in B_t\}$ is nonempty and
connected in $T$.  We write $\TD(G)$ for the set of tree
decompositions of $G$, and for $h:2^V\to\Rp$ we put
\[
\hwd_h(G):=\min_{(T,B)\in\TD(G)}\ \max_{t\in V(T)}h(B_t).
\]
A function $h:2^V\to\Rp$ is \emph{edge dominated} on $G$ if $h(e)\le1$
for every $e\in E$.  We write $ED(G)$ for the edge dominated functions
on $G$, and define
\begin{align*}
\adw(G):=&\sup_{h\in M_n \cap ED(G)} \hwd_h(G) &&\text{adaptive width}\\
\normw(G):=&\sup_{h\in N_n \cap ED(G)} \hwd_h(G) &&\text{normal width}\\
\linw(G):=&\sup_{h\in\text{Lin}_n \cap ED(G)} \hwd_h(G) &&\text{linear width}\\
\entw(G):=&\sup_{h\in\bar\Gamma_n^* \cap ED(G)} \hwd_h(G) &&\text{entropic width}\\
\subw(G):=&\sup_{h\in\Gamma_n \cap ED(G)} \hwd_h(G) &&\text{submodular width}
\end{align*}
Since
$M_n\subseteq N_n\subseteq\text{Lin}_n\subseteq\bar\Gamma^*_n\subseteq\Gamma_n$
we always have
\[
\adw(G)\le\normw(G)\le\linw(G)\le\entw(G)\le\subw(G).
\]
Edge domination can always be enforced by normalization.
For any nonzero polymatroid $h$, let
$c_h:=\max_{e\in E}h(e)$. Since $G$ has no isolated vertices,
$c_h>0$: otherwise monotonicity forces every singleton value to
vanish, and subadditivity gives $h=0$. The polymatroid $h/c_h$
is edge dominated. Since $\hwd_h(G)$ is positively homogeneous
in $h$, for every nontrivial class $\mathcal C$ of polymatroids
closed under multiplication by positive scalars,
\[
\sup_{h\in\mathcal C\cap ED(G)}\hwd_h(G)
=
\sup_{h\in\mathcal C\setminus\{0\}}
\frac{\hwd_h(G)}{\max_{e\in E}h(e)}.
\]
Thus we may construct a polymatroid first and normalize it afterwards.

By continuity, taking the closures in the definitions of linear
and entropic polymatroids does not change the normalized suprema.
Furthermore, since $G$ has no isolated vertices, edge domination
implies $h(U)\le|U|$ for every $U\subseteq V(G)$.
The five classes used in the width definitions are closed, so their
edge dominated subsets are compact. Continuity of $\hwd_h(G)$
therefore shows that all five suprema are attained.

\begin{theorem}\label{thm:main}
  There is a graph $G$ on $32$ vertices such that
\[
\boxed{\begin{aligned}
\adw(G)=6<\normw(G)=\frac{69}{11}<\linw(G)=\frac{44}{7}
<\frac{635}{101}\le\entw(G)
\le\frac{183}{29}<\subw(G)=\frac{19}{3}.
\end{aligned}}
\]
\end{theorem}

The graph $G$ mentioned in the theorem is defined in Section~\ref{sec:graph} and shown in
Figure~\ref{fig:graph32}.
In particular, for this graph
\[
\adw(G)<\normw(G)<\linw(G)<\entw(G)<\subw(G).
\]
Notice that the theorem gives exact values for four out of the five
widths.  We leave open the exact value of $\entw(G)$, but notice that
it belongs to the interval $[635/101,\,183/29]$.

\section{Linear Inequalities}
\label{sec:inequalities}

We consider linear inequalities of the form
$\sum_{U \subseteq \bm X}c_U h(U) \geq 0$.  Often we express these
inequalities in terms of the following quantities:
\begin{align*}
H_h(A\mid C)=& h(AC)-h(C) & I_h(A;B\mid C)=&h(AC)+h(BC)-h(ABC)-h(C)
\end{align*}

Throughout this section $A,B,C,X,Y$ denote arbitrary subsets of
$\bm X$.
From Section~\ref{sec:core-bounds} on, those inequalities are only
ever instantiated at single central variables.

\begin{description}
\item[Inequalities in $\Gamma_n$.] An inequality satisfied by all
  polymatroids is called a \emph{Shannon inequality}.
  Inequalities~\eqref{eq:normalized}, \eqref{eq:monotone}, and
  \eqref{eq:submodular} are the \emph{Basic Shannon Inequalities}.
  Consider a graph $G = (V(G), E(G))$, and a set of weights $w_e\ge0$
  for $e\in E$.  Fix a set $S \subseteq V(G)$ such that
  $\sum_{e\ni v}w_e\ge 1$ for every $v\in S$.  Then we say that $w$ is
  a \emph{fractional edge cover} of $S$, and one can prove that the
  following is a Shannon inequality:
\begin{align}\label{eq:cover}
  h(S) \le {} & \sum_e w_e\,h(e)
\end{align}
\item[Inequalities in $\bar \Gamma_n^*$.] If an inequality holds for
  all entropic functions $h \in\Gamma_n^*$, then it also holds for all
  almost entropic functions $h \in \bar\Gamma_n^*$.  For that reason
  we consider only the latter in our discussion. The following is
  called the Zhang--Yeung inequality~\cite{zhang1998characterization}
  and holds for all $h \in \bar \Gamma_n^*$:
\begin{align}
I_h(X;Y)  \le {} &I_h(X;Y\mid A)+I_h(X;Y\mid B)+I_h(A;B)\notag\\
               &{}+I_h(X;Y\mid A)+I_h(A;Y\mid X)+I_h(A;X\mid Y).
\label{eq:zy}
\end{align}
The term $I_h(X;Y\mid A)$ occurs twice, and we have deliberately not
collected the two copies into a single $2I_h(X;Y\mid A)$: the first
line of \eqref{eq:zy} is precisely Ingleton's inequality
\eqref{eq:ing} below, and the second line is the surplus that
\eqref{eq:zy} carries over it, as recorded in \eqref{eq:zy-vs-ing}.
Note also that \eqref{eq:zy} is not symmetric in $A$ and $B$; this is
why Lemma~\ref{lem:zy-core} has to apply it twice, once in each order.
We write $\ZY_h(X,Y;A,B)$ for the slack of \eqref{eq:zy}, i.e.\ the right-hand side
minus the left-hand side of \eqref{eq:zy}, so that
\eqref{eq:zy} reads $\ZY_h(X,Y;A,B)\ge0$.  We note that no general
characterization of all entropic inequalities is known.

\item[Inequalities in $\text{Lin}_n$.] The following inequality holds
  for all linear polymatroids $h \in \text{Lin}_n$, and is called
  \emph{Ingleton's inequality}~\cite{ingleton1971representation}:
\begin{align}\label{eq:ing}
  I_h(X;Y) \le {} & I_h(X;Y\mid A)+I_h(X;Y\mid B)+I_h(A;B).
\end{align}
The inequality holds for rank functions, and hence for their
nonnegative scalar multiples and limits by continuity.
We write $\Ing_h(X,Y;A,B)$ for the slack, i.e.\ the right-hand side
minus the left-hand side of \eqref{eq:ing}, so that \eqref{eq:ing}
reads $\Ing_h(X,Y;A,B)\ge0$.  We note that no general
characterization of all linear rank inequalities is known; see
\cite{DBLP:journals/corr/abs-0910-0284} for the state of the art on
five or more variables.

\item[Inequalities in $N_n$.] Write
  $\check h(A) := h(\bm X)-h(\bm X\setminus A)$.  For any nonempty
  $K \subseteq \bm X$, a normal polymatroid $h \in N_n$ satisfies:
  \begin{align}
    \sum_{A \subseteq K}(-1)^{|K \setminus A|}\check h(A) \geq {} & 0
    \label{eq:normal:polymatroids}
  \end{align}
  Indeed, for $h=c\cdot|\varphi(\cdot)|$ the left-hand side is $c$
  times the number of $\omega\in\Omega$ lying in $\varphi(X_i)$ for
  exactly the $X_i\in K$, so \eqref{eq:normal:polymatroids} just says
  that these counts are nonnegative.  One can also show that these
  inequalities completely characterize the normal polymatroids: $h \in
  \Gamma_n$ is normal iff it satisfies all inequalities of the
  form~\eqref{eq:normal:polymatroids}.

\item[Inequalities in $M_n$.] Every modular polymatroid $h\in M_n$
  satisfies:
  \begin{align}
    h(A)+h(B) \le {} & h(A\cup B)+h(A\cap B).
    \label{eq:modular:polymatroids}
  \end{align}
  Together with submodularity, this inequality holds with equality.
  Conversely, a polymatroid $h\in\Gamma_n$ is modular iff it satisfies
  \eqref{eq:modular:polymatroids} for all $A,B\subseteq\bm X$.

\end{description}

It is useful to compare Equations~\eqref{eq:zy} and \eqref{eq:ing}.
$\ZY_h(X,Y;A,B) \geq \Ing_h(X,Y;A,B)$ for any polymatroid $h$, more precisely:
\begin{equation}\label{eq:zy-vs-ing}
\ZY_h(X,Y;A,B)-\Ing_h(X,Y;A,B) = I_h(X;Y\mid A)+I_h(A;Y\mid X)+I_h(A;X\mid Y).
\end{equation}

\section{Three Special Polymatroids}

\label{sec:three:special}

To separate two width notions, it is not enough to exhibit a polymatroid
belonging to one class but not the other. Instead we need a graph on which this additional freedom yields a strictly larger width. The following examples introduce the ingredients for our lower-bound constructions. The linear example illustrates a dependence pattern that normal polymatroids cannot represent. The entropic example violates Ingleton's inequality and will serve as a building block for the entropic lower bound. Finally, the
graph-induced polymatroid will supply the submodular lower bound. In the
subsequent sections, we use these ideas to construct witnesses on the
central graph, then extend those witnesses to the separating graph and
prove that every tree decomposition contains a bag of sufficiently large
value.

The following notation will be convenient. Given a set $\bm X = \set{X_1, X_2, \ldots}$ we abbreviate a subset of
them by the string of its indices, for example $\xs{123}$ stands for
$\set{X_1,X_2,X_3}$.

\paragraph{A Linear Polymatroid.} Consider the following coverage function
from $\set{X_1,X_2,X_3}$ to vectors in the vector space $\R^2$:
\begin{align*}
  \varphi(X_1)= & \set{(1,0)} & \varphi(X_2) = & \set{(0,1)} &
  \varphi(X_3)= & \set{(1,1)}
\end{align*}
Let $h$ be its rank function.  More precisely:
\begin{align*}
  h(\emptyset)=&0 & h(X_1)=h(X_2)=h(X_3)=&1 & h(\xs{12})=h(\xs{13})=h(\xs{23})=h(\xs{123})=2
\end{align*}
This is a linear polymatroid which is not normal, because it violates
inequality~\eqref{eq:normal:polymatroids} for $K=\xs{123}$:
\begin{align*}
  h(X_1)+h(X_2)+h(X_3)-h(\xs{12})-h(\xs{13})-h(\xs{23})+h(\xs{123})=&-1
\end{align*}
We will use a linear polymatroid in Section~\ref{subsec:shannon-core},
where the same argument shows that it is not a normal polymatroid.

\paragraph{An Entropic Polymatroid.} The following entropic polymatroid is a
building block for Section~\ref{subsec:zy-core} below.  Consider the
following probability distribution on 4 variables
$X_1, X_2, X_3, X_4$, computed from 2 independent fair Bernoulli random
variables $R, S$:

\begin{align*}
  &
    \begin{array}{|cc||c|c|c|c|c} \cline{1-6}
R & S &  X_1 := R\wedge S&X_2:=R\vee S&X_3:=R&X_4:=S&p \\ \cline{1-6}
0 & 0 & 0 & 0 & 0 & 0 & 1/4 \\
0 & 1 & 0 & 1 & 0 & 1 & 1/4 \\
1 & 0 & 0 & 1 & 1 & 0 & 1/4 \\
1 & 1 & 1 & 1 & 1 & 1 & 1/4 \\ \cline{1-6}
    \end{array}
\end{align*}

Let $h_0$ be the entropy function of these four variables.  For
example,
$h_0(X_1)=h_0(X_2) = \frac{1}{4}\log_2 4 + \frac{3}{4}\log_2 (4/3) =
2-\frac{3}{4}\log_2 3$, $h_0(X_3)=h_0(X_4)=1$, etc.  Notice that $h_0$
is not a linear polymatroid, because it fails Ingleton's inequality.  To
see this, instantiate \eqref{eq:ing} at $(X,Y,A,B)=(X_1,X_2,X_3,X_4)$,
which reads
\[
I_{h_0}(X_1;X_2)\ \le\
I_{h_0}(X_1;X_2\mid X_3)+I_{h_0}(X_1;X_2\mid X_4)+I_{h_0}(X_3;X_4).
\]
All three terms on the right vanish.  Indeed $I_{h_0}(X_3;X_4)=0$
because $R$ and $S$ are independent; and conditioned on $X_3=R$ one of
$X_1,X_2$ is constant, since $R=0$ forces $X_1=0$ while $R=1$ forces
$X_2=1$, so $I_{h_0}(X_1;X_2\mid X_3)=0$, and $I_{h_0}(X_1;X_2\mid
X_4)=0$ by the same argument with $S$ in place of $R$.  The left-hand
side, however, is
\[
I_{h_0}(X_1;X_2)=h_0(X_1)+h_0(X_2)-h_0(\xs{12})
 =2\Bigl(2-\frac34\log_2 3\Bigr)-\frac32
 =\frac52-\frac32\log_2 3\approx0.1226>0 .
\]
Hence $\Ing_{h_0}(X_1,X_2;X_3,X_4)\approx-0.1226<0$, so $h_0$ is not a
linear polymatroid.

\paragraph{A Polymatroid Induced by a Graph.} Fix any graph $G = (V,E)$ and
define the following function $f_G: 2^V \rightarrow \Rp$:
\begin{align}
  f_G(U) := &
            \begin{cases}
              0 & \text{when } U=\emptyset\\
              2 & \text{when } |U|=1 \\
              3 & \text{when } U \in E\\
              4 & \text{otherwise}
            \end{cases} \label{eq:function:f:g}
\end{align}
In words, $f_G$ is 2 on singleton sets, is 3 on edges, and 4
everywhere else.  One can check that $f_G$ is monotone and submodular,
thus $f_G$ is a polymatroid.  Zhang and Yeung showed that $f_G$, where
$G = K_4 - \set{34}$ (the complete graph on 4 vertices, without one
edge), does not satisfy inequality~\eqref{eq:zy}, thus $f_G$ is not
entropic: see Fig.5 in~\cite{DBLP:conf/lics/Suciu23} for a
representation of $f_G$.  The polymatroid we use in
Section~\ref{subsec:subw-core} is $f_G$ for $K_5-\set{15,25,34}$.

\section{The Separating Graph}
\label{sec:graph}

The Central Graph $G_0$ is $K_5-\set{15,25,34}$.  On this graph
we prove several upper and lower bounds for the different classes of
polymatroids.  However, $G_0$ is insufficient to separate the
corresponding notions of widths.  The Big Graph, $G$, extends $G_0$
with scaffolding that ensures that the separations on $G_0$ extend to
separations of corresponding notions of widths on $G$.

{\bf The Central Graph.} Let $G_0 = (X, E_0)$, where
\begin{align*}
  X = {}& \set{X_1, X_2, X_3, X_4, X_5} \\
  E_0={}& \set{\xs{12},\xs{13},\xs{14},\xs{23},\xs{24},\xs{35},\xs{45}}.
\end{align*}

We will often refer to the five nodes $X_1,\ldots,X_5$ as variables.
A key role will be played by the following five sets of variables:
\begin{align}
  (U_1,U_2,U_3,U_4,U_5)= {}& (\xs{123},\xs{124},\xs{34},\xs{15},\xs{25})\label{eq:the:sets:u}
\end{align}
The sets $U_i$ are exactly the minimal subsets of $X$ that are not
contained in any edge of $E_0$: there are three non-edges
$\xs{34},\xs{15},\xs{25}$, and two triangles $\xs{123}$ and
$\xs{124}$.

{\bf The Big Graph $G$.} The big graph $G$ has $G_0$ as an induced
subgraph, and is defined as follows.  Let $A_i=\set{a_{i,1},a_{i,2}}$,
$i=1,\ldots,5$, be five disjoint pairs of \emph{anchor nodes}, all
disjoint from $X$; we call each pair $A_i$ an \emph{anchor group}.  We
set $A=\bigcup_i A_i$.  Join every pair of anchors from distinct
groups, so that $G[A]$ is the complete multipartite graph
$K_{2,2,2,2,2}$ whose five parts are the anchor groups: the two
anchors of one group are \emph{non}-adjacent, while every anchor is
adjacent to all eight anchors outside its own group.  Set
\[
(r_1,r_2,r_3,r_4,r_5)=(3,3,3,4,4).
\]
For each $i=1,\ldots,5$, let $p_{i,1},\ldots,p_{i,r_i}$ be $r_i$
\emph{connector nodes}, each adjacent to every vertex of
$A_i\cup U_i$.

The graph $G$ consists of the central graph $G_0$, the anchor graph
$G[A]$, and the connector nodes and their edges.
Figure~\ref{fig:graph32} shows the complete graph.  The number of
nodes and edges are:
\begin{align*}
  |V(G)|=&5+10+17=32 \\
  |E(G)|=&7+40+74=121
\end{align*}

\begin{figure}[p]
\centering
\resizebox{\textwidth}{!}{%
  \begin{tikzpicture}[x=1cm,y=1cm,
      cnode/.style={circle,draw=black,line width=.9pt,fill=white,
                    inner sep=0pt,minimum size=7.4mm,font=\small},
      anode/.style={circle,draw=#1!80!black,fill=#1,inner sep=0pt,
                    minimum size=2.6mm},
      pnode/.style={circle,draw=#1!85!black,line width=.6pt,fill=white,
                    inner sep=0pt,minimum size=2.4mm},
      vlab/.style={font=\tiny,text=black,inner sep=.6pt,
                   fill=white,fill opacity=.5,text opacity=1},
      glab/.style={font=\small,text=#1!72!black,align=center},
    ]
    \coordinate (x1) at (54.00:1.750);
    \coordinate (x2) at (126.00:1.750);
    \coordinate (x3) at (198.00:1.750);
    \coordinate (x5) at (270.00:1.750);
    \coordinate (x4) at (342.00:1.750);
    \coordinate (a21) at (43.00:4.950);
    \coordinate (a22) at (65.00:4.950);
    \coordinate (p21) at (39.00:3.550);
    \coordinate (p22) at (54.00:3.550);
    \coordinate (p23) at (69.00:3.550);
    \coordinate (a11) at (115.00:4.950);
    \coordinate (a12) at (137.00:4.950);
    \coordinate (p11) at (111.00:3.550);
    \coordinate (p12) at (126.00:3.550);
    \coordinate (p13) at (141.00:3.550);
    \coordinate (a51) at (187.00:4.950);
    \coordinate (a52) at (209.00:4.950);
    \coordinate (p51) at (179.00:3.550);
    \coordinate (p52) at (191.67:3.550);
    \coordinate (p53) at (204.33:3.550);
    \coordinate (p54) at (217.00:3.550);
    \coordinate (a31) at (259.00:4.950);
    \coordinate (a32) at (281.00:4.950);
    \coordinate (p31) at (255.00:3.550);
    \coordinate (p32) at (270.00:3.550);
    \coordinate (p33) at (285.00:3.550);
    \coordinate (a41) at (331.00:4.950);
    \coordinate (a42) at (353.00:4.950);
    \coordinate (p41) at (323.00:3.550);
    \coordinate (p42) at (335.67:3.550);
    \coordinate (p43) at (348.33:3.550);
    \coordinate (p44) at (1.00:3.550);
    \begin{scope}[black!22,line width=.32pt]
      \draw (a21) -- (43.00:5.450) arc[start angle=43.00, end angle=115.00, radius=5.450] -- (a11); %
      \draw (a21) -- (43.00:5.472) arc[start angle=43.00, end angle=137.00, radius=5.472] -- (a12); %
      \draw (a21) -- (43.00:5.494) arc[start angle=43.00, end angle=187.00, radius=5.494] -- (a51); %
      \draw (a21) -- (43.00:5.516) arc[start angle=43.00, end angle=209.00, radius=5.516] -- (a52); %
      \draw (a31) -- (259.00:5.538) arc[start angle=259.00, end angle=403.00, radius=5.538] -- (a21); %
      \draw (a32) -- (281.00:5.560) arc[start angle=281.00, end angle=403.00, radius=5.560] -- (a21); %
      \draw (a41) -- (331.00:5.582) arc[start angle=331.00, end angle=403.00, radius=5.582] -- (a21); %
      \draw (a42) -- (353.00:5.604) arc[start angle=353.00, end angle=403.00, radius=5.604] -- (a21); %
      \draw (a22) -- (65.00:5.626) arc[start angle=65.00, end angle=115.00, radius=5.626] -- (a11); %
      \draw (a22) -- (65.00:5.648) arc[start angle=65.00, end angle=137.00, radius=5.648] -- (a12); %
      \draw (a22) -- (65.00:5.670) arc[start angle=65.00, end angle=187.00, radius=5.670] -- (a51); %
      \draw (a22) -- (65.00:5.692) arc[start angle=65.00, end angle=209.00, radius=5.692] -- (a52); %
      \draw (a31) -- (259.00:5.714) arc[start angle=259.00, end angle=425.00, radius=5.714] -- (a22); %
      \draw (a32) -- (281.00:5.736) arc[start angle=281.00, end angle=425.00, radius=5.736] -- (a22); %
      \draw (a41) -- (331.00:5.758) arc[start angle=331.00, end angle=425.00, radius=5.758] -- (a22); %
      \draw (a42) -- (353.00:5.780) arc[start angle=353.00, end angle=425.00, radius=5.780] -- (a22); %
      \draw (a11) -- (115.00:5.802) arc[start angle=115.00, end angle=187.00, radius=5.802] -- (a51); %
      \draw (a11) -- (115.00:5.824) arc[start angle=115.00, end angle=209.00, radius=5.824] -- (a52); %
      \draw (a11) -- (115.00:5.846) arc[start angle=115.00, end angle=259.00, radius=5.846] -- (a31); %
      \draw (a11) -- (115.00:5.868) arc[start angle=115.00, end angle=281.00, radius=5.868] -- (a32); %
      \draw (a41) -- (331.00:5.890) arc[start angle=331.00, end angle=475.00, radius=5.890] -- (a11); %
      \draw (a42) -- (353.00:5.912) arc[start angle=353.00, end angle=475.00, radius=5.912] -- (a11); %
      \draw (a12) -- (137.00:5.934) arc[start angle=137.00, end angle=187.00, radius=5.934] -- (a51); %
      \draw (a12) -- (137.00:5.956) arc[start angle=137.00, end angle=209.00, radius=5.956] -- (a52); %
      \draw (a12) -- (137.00:5.978) arc[start angle=137.00, end angle=259.00, radius=5.978] -- (a31); %
      \draw (a12) -- (137.00:6.000) arc[start angle=137.00, end angle=281.00, radius=6.000] -- (a32); %
      \draw (a41) -- (331.00:6.022) arc[start angle=331.00, end angle=497.00, radius=6.022] -- (a12); %
      \draw (a42) -- (353.00:6.044) arc[start angle=353.00, end angle=497.00, radius=6.044] -- (a12); %
      \draw (a51) -- (187.00:6.066) arc[start angle=187.00, end angle=259.00, radius=6.066] -- (a31); %
      \draw (a51) -- (187.00:6.088) arc[start angle=187.00, end angle=281.00, radius=6.088] -- (a32); %
      \draw (a51) -- (187.00:6.110) arc[start angle=187.00, end angle=331.00, radius=6.110] -- (a41); %
      \draw (a51) -- (187.00:6.132) arc[start angle=187.00, end angle=353.00, radius=6.132] -- (a42); %
      \draw (a52) -- (209.00:6.154) arc[start angle=209.00, end angle=259.00, radius=6.154] -- (a31); %
      \draw (a52) -- (209.00:6.176) arc[start angle=209.00, end angle=281.00, radius=6.176] -- (a32); %
      \draw (a52) -- (209.00:6.198) arc[start angle=209.00, end angle=331.00, radius=6.198] -- (a41); %
      \draw (a52) -- (209.00:6.220) arc[start angle=209.00, end angle=353.00, radius=6.220] -- (a42); %
      \draw (a31) -- (259.00:6.242) arc[start angle=259.00, end angle=331.00, radius=6.242] -- (a41); %
      \draw (a31) -- (259.00:6.264) arc[start angle=259.00, end angle=353.00, radius=6.264] -- (a42); %
      \draw (a32) -- (281.00:6.286) arc[start angle=281.00, end angle=331.00, radius=6.286] -- (a41); %
      \draw (a32) -- (281.00:6.308) arc[start angle=281.00, end angle=353.00, radius=6.308] -- (a42); %
    \end{scope}
    \begin{scope}[grpII,line width=.42pt,opacity=.8]
      \draw (p21) -- (a21);
      \draw (p21) -- (a22);
      \draw (p21) -- (x1);
      \draw (p21) -- (x2);
      \draw (p21) -- (x4);
      \draw (p22) -- (a21);
      \draw (p22) -- (a22);
      \draw (p22) -- (x1);
      \draw (p22) -- (x2);
      \draw (p22) -- (x4);
      \draw (p23) -- (a21);
      \draw (p23) -- (a22);
      \draw (p23) -- (x1);
      \draw (p23) -- (x2);
      \draw (p23) -- (x4);
    \end{scope}
    \begin{scope}[grpI,line width=.42pt,opacity=.8]
      \draw (p11) -- (a11);
      \draw (p11) -- (a12);
      \draw (p11) -- (x1);
      \draw (p11) -- (x2);
      \draw (p11) -- (x3);
      \draw (p12) -- (a11);
      \draw (p12) -- (a12);
      \draw (p12) -- (x1);
      \draw (p12) -- (x2);
      \draw (p12) -- (x3);
      \draw (p13) -- (a11);
      \draw (p13) -- (a12);
      \draw (p13) -- (x1);
      \draw (p13) -- (x2);
      \draw (p13) -- (x3);
    \end{scope}
    \begin{scope}[grpV,line width=.42pt,opacity=.8]
      \draw (p51) -- (a51);
      \draw (p51) -- (a52);
      \draw (p51) -- (x2);
      \draw (p51) -- (x5);
      \draw (p52) -- (a51);
      \draw (p52) -- (a52);
      \draw (p52) -- (x2);
      \draw (p52) -- (x5);
      \draw (p53) -- (a51);
      \draw (p53) -- (a52);
      \draw (p53) -- (x2);
      \draw (p53) -- (x5);
      \draw (p54) -- (a51);
      \draw (p54) -- (a52);
      \draw (p54) -- (x2);
      \draw (p54) -- (x5);
    \end{scope}
    \begin{scope}[grpIII,line width=.42pt,opacity=.8]
      \draw (p31) -- (a31);
      \draw (p31) -- (a32);
      \draw (p31) -- (x3);
      \draw (p31) -- (x4);
      \draw (p32) -- (a31);
      \draw (p32) -- (a32);
      \draw (p32) -- (x3);
      \draw (p32) -- (x4);
      \draw (p33) -- (a31);
      \draw (p33) -- (a32);
      \draw (p33) -- (x3);
      \draw (p33) -- (x4);
    \end{scope}
    \begin{scope}[grpIV,line width=.42pt,opacity=.8]
      \draw (p41) -- (a41);
      \draw (p41) -- (a42);
      \draw (p41) -- (x1);
      \draw (p41) -- (x5);
      \draw (p42) -- (a41);
      \draw (p42) -- (a42);
      \draw (p42) -- (x1);
      \draw (p42) -- (x5);
      \draw (p43) -- (a41);
      \draw (p43) -- (a42);
      \draw (p43) -- (x1);
      \draw (p43) -- (x5);
      \draw (p44) -- (a41);
      \draw (p44) -- (a42);
      \draw (p44) -- (x1);
      \draw (p44) -- (x5);
    \end{scope}
    \begin{scope}[black,line width=1.35pt]
      \draw (x1) -- (x2);
      \draw (x1) -- (x3);
      \draw (x1) -- (x4);
      \draw (x2) -- (x3);
      \draw (x2) -- (x4);
      \draw (x3) -- (x5);
      \draw (x4) -- (x5);
    \end{scope}
    \node[anode=grpII] (A21) at (a21) {};
    \node[anode=grpII] (A22) at (a22) {};
    \node[pnode=grpII] (P21) at (p21) {};
    \node[pnode=grpII] (P22) at (p22) {};
    \node[pnode=grpII] (P23) at (p23) {};
    \node[anode=grpI] (A11) at (a11) {};
    \node[anode=grpI] (A12) at (a12) {};
    \node[pnode=grpI] (P11) at (p11) {};
    \node[pnode=grpI] (P12) at (p12) {};
    \node[pnode=grpI] (P13) at (p13) {};
    \node[anode=grpV] (A51) at (a51) {};
    \node[anode=grpV] (A52) at (a52) {};
    \node[pnode=grpV] (P51) at (p51) {};
    \node[pnode=grpV] (P52) at (p52) {};
    \node[pnode=grpV] (P53) at (p53) {};
    \node[pnode=grpV] (P54) at (p54) {};
    \node[anode=grpIII] (A31) at (a31) {};
    \node[anode=grpIII] (A32) at (a32) {};
    \node[pnode=grpIII] (P31) at (p31) {};
    \node[pnode=grpIII] (P32) at (p32) {};
    \node[pnode=grpIII] (P33) at (p33) {};
    \node[anode=grpIV] (A41) at (a41) {};
    \node[anode=grpIV] (A42) at (a42) {};
    \node[pnode=grpIV] (P41) at (p41) {};
    \node[pnode=grpIV] (P42) at (p42) {};
    \node[pnode=grpIV] (P43) at (p43) {};
    \node[pnode=grpIV] (P44) at (p44) {};
    \node[cnode] at (x1) {$X_1$};
    \node[cnode] at (x2) {$X_2$};
    \node[cnode] at (x3) {$X_3$};
    \node[cnode] at (x5) {$X_5$};
    \node[cnode] at (x4) {$X_4$};
    \node[vlab=grpII] at (43.00:5.280) {$a_{2,1}$};
    \node[vlab=grpII] at (65.00:5.280) {$a_{2,2}$};
    \node[vlab=grpII] at (39.00:3.930) {$p_{2,1}$};
    \node[vlab=grpII] at (54.00:3.930) {$p_{2,2}$};
    \node[vlab=grpII] at (69.00:3.930) {$p_{2,3}$};
    \node[glab=grpII] at (54.00:7.150) {$A_2$\\[-1pt]$U_2=\xs{124}$};
    \node[vlab=grpI] at (115.00:5.280) {$a_{1,1}$};
    \node[vlab=grpI] at (137.00:5.280) {$a_{1,2}$};
    \node[vlab=grpI] at (111.00:3.930) {$p_{1,1}$};
    \node[vlab=grpI] at (126.00:3.930) {$p_{1,2}$};
    \node[vlab=grpI] at (141.00:3.930) {$p_{1,3}$};
    \node[glab=grpI] at (126.00:7.150) {$A_1$\\[-1pt]$U_1=\xs{123}$};
    \node[vlab=grpV] at (187.00:5.280) {$a_{5,1}$};
    \node[vlab=grpV] at (209.00:5.280) {$a_{5,2}$};
    \node[vlab=grpV] at (179.00:3.930) {$p_{5,1}$};
    \node[vlab=grpV] at (191.67:3.930) {$p_{5,2}$};
    \node[vlab=grpV] at (204.33:3.930) {$p_{5,3}$};
    \node[vlab=grpV] at (217.00:3.930) {$p_{5,4}$};
    \node[glab=grpV] at (198.00:7.150) {$A_5$\\[-1pt]$U_5=\xs{25}$};
    \node[vlab=grpIII] at (259.00:5.280) {$a_{3,1}$};
    \node[vlab=grpIII] at (281.00:5.280) {$a_{3,2}$};
    \node[vlab=grpIII] at (255.00:3.930) {$p_{3,1}$};
    \node[vlab=grpIII] at (270.00:3.930) {$p_{3,2}$};
    \node[vlab=grpIII] at (285.00:3.930) {$p_{3,3}$};
    \node[glab=grpIII] at (270.00:7.150) {$A_3$\\[-1pt]$U_3=\xs{34}$};
    \node[vlab=grpIV] at (331.00:5.280) {$a_{4,1}$};
    \node[vlab=grpIV] at (353.00:5.280) {$a_{4,2}$};
    \node[vlab=grpIV] at (323.00:3.930) {$p_{4,1}$};
    \node[vlab=grpIV] at (335.67:3.930) {$p_{4,2}$};
    \node[vlab=grpIV] at (348.33:3.930) {$p_{4,3}$};
    \node[vlab=grpIV] at (1.00:3.930) {$p_{4,4}$};
    \node[glab=grpIV] at (342.00:7.150) {$A_4$\\[-1pt]$U_4=\xs{15}$};
    \begin{scope}[shift={(0,-8.05)},font=\scriptsize]
      \node[cnode,minimum size=3.4mm,font=\tiny] at (-4.55,0) {};
      \node[anchor=west] at (-4.32,0) {central vertex};
      \node[anode=grpN] at (-1.75,0) {};
      \node[anchor=west] at (-1.58,0) {anchor};
      \node[pnode=grpN] at (0.25,0) {};
      \node[anchor=west] at (0.42,0) {connector};
      \draw[black!30,line width=.32pt] (2.62,0) -- (3.32,0);
      \node[anchor=west] at (3.42,0) {anchor edge};
    \end{scope}
  \end{tikzpicture}}
\caption{The 32-vertex graph $G$, drawn with all $121$ edges.
The five central variables $X_1,\ldots,X_5$ sit on the inner pentagon.  Its
five sides, together with two of its five diagonals, are exactly the
seven central edges $E_0$.  The other three diagonals, $\xs{34}$,
$\xs{15}$ and $\xs{25}$, are absent from $G$.  They are therefore
minimal subsets of $X$ contained in no central edge, and together with
the only two triangles of $G_0$, namely $\xs{123}$ and $\xs{124}$, they
make up the five sets $U_1,\ldots,U_5$ of Section~\ref{sec:graph}.
Filled colored vertices are the ten anchors $a_{i,k}$ and hollow ones
the $17$ connectors $p_{i,k}$; each connector of group $i$ is joined to
both anchors of $A_i$ and to every central variable in $U_i$.
The $40$ light arcs join every pair of anchors from distinct groups,
forming $K_{2,2,2,2,2}$; they are routed outside the anchor circle, so
every arc that appears inside it is a connector edge.
Edge crossings are not vertices.}
\label{fig:graph32}
\end{figure}
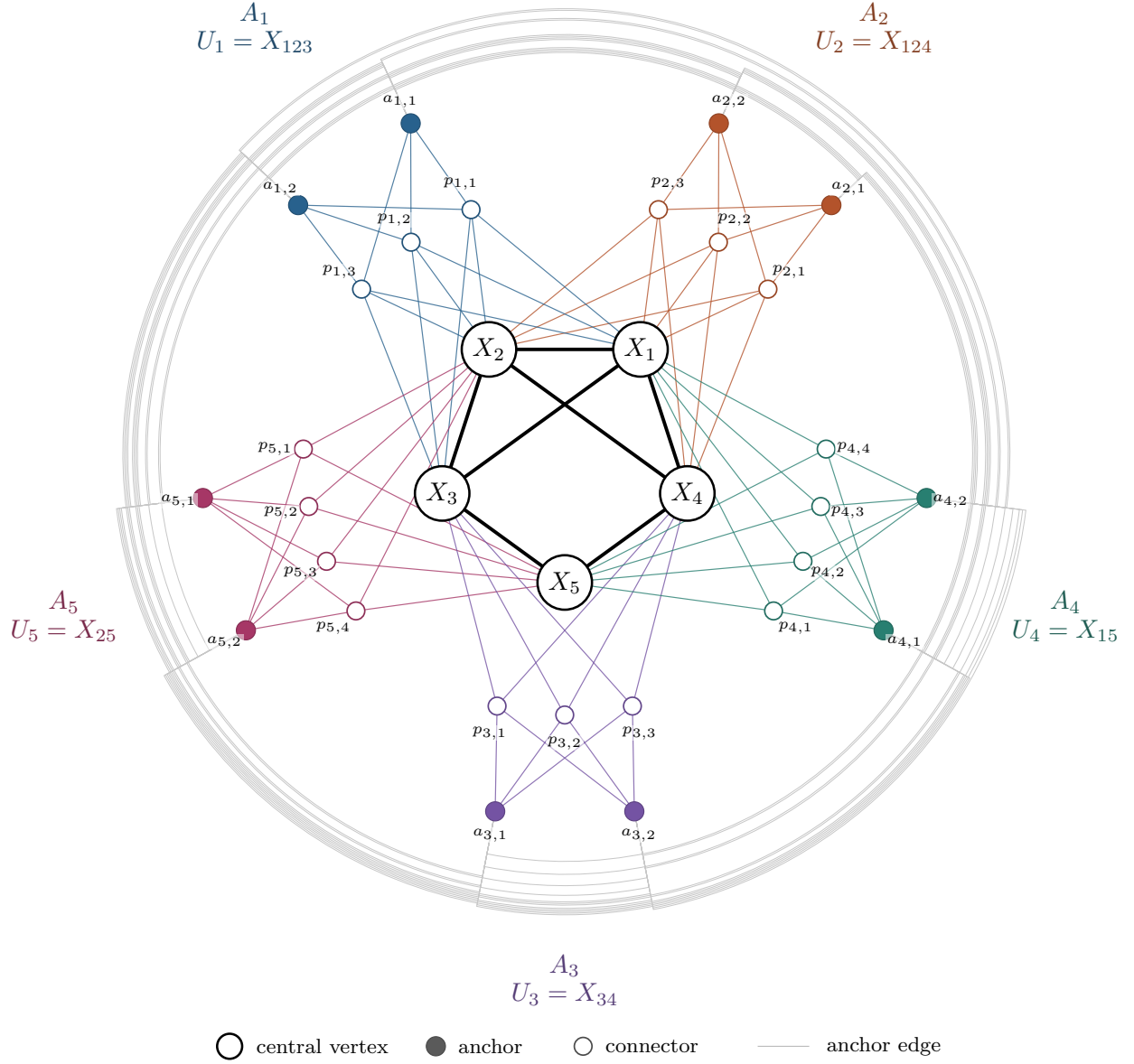

\section{The Central Graph}
\label{sec:core-bounds}

We start by considering only the central graph $G_0=(X, E_0)$ and the
sets $U_1,\ldots,U_5$ defined in~\eqref{eq:the:sets:u}.  A
``polymatroid on $G_0$'' means a polymatroid with variables $X$.  We
establish here several inequalities for polymatroids on $G_0$: they do
not refer to tree decompositions of $G_0$, but instead are bounds on
the following quantity:
\begin{align}
 M(h) := & \min(h(U_1), h(U_2), h(U_3), h(U_4), h(U_5)) \label{eq:mh}
\end{align}
In the next section we will use these bounds to compute bounds on the
widths of the large graph $G$. 

In this section we prove the following:

\begin{theorem}\label{thm:central}
  With $M$ as above, the following hold:
\begin{align}
\max_{h \in M_X\cap ED(G_0)} M(h) &=1 &
\max_{h \in N_X\cap ED(G_0)} M(h) &=\frac{14}{11}  \notag\\
\max_{h \in \text{Lin}_X\cap ED(G_0)} M(h)&=\frac{9}{7} &
\frac{130}{101}\leq\max_{h \in \bar \Gamma^*_X\cap ED(G_0)} M(h)&\leq\frac{38}{29} \notag\\
\max_{h \in \Gamma_X\cap ED(G_0)} M(h)&=\frac{4}{3} &&
\label{eq:central-bounds}
\end{align}
In particular:
\begin{align}
\max_{h \in M_X\cap ED(G_0)} M(h)
&<\max_{h \in N_X\cap ED(G_0)} M(h)
<\max_{h \in \text{Lin}_X\cap ED(G_0)} M(h)\notag\\
&<\max_{h \in \bar \Gamma^*_X\cap ED(G_0)} M(h)
<\max_{h \in \Gamma_X\cap ED(G_0)} M(h)
\label{eq:central-bounds:2}
\end{align}
\end{theorem}

We prove the theorem in the rest of the section. The upper bounds
$1$, $14/11$, $9/7$, $38/29$, and $4/3$ are
Lemmas~\ref{lem:modular-core}, \ref{lem:normal-core},
\ref{lem:ingleton}, \ref{lem:zy-core}, and~\ref{lem:subw-core}.
The modular bound is attained by $h(U)=|U|/2$, as shown in
Lemma~\ref{lem:modular-core}.  The
matching lower bounds $14/11$, $9/7$, and $4/3$ are
Lemmas~\ref{lem:normal-tight}, \ref{lem:ingleton-tight}
and~\ref{lem:subw-tight}, and the lower bound $130/101$ is
Lemma~\ref{lem:zy-tight}.  The
inequalities~\eqref{eq:central-bounds:2} follow by $1<14/11$, $14/11<9/7$
(because $14\cdot7=98<99=11\cdot9$), $9/7<130/101$ (because
$9\cdot101=909<910=7\cdot130$), and $38/29<4/3$ (because
$38\cdot3=114<116=29\cdot4$).

{\bf Discussion.} The theorem immediately implies that the entropic
bound differs from the polymatroid bound for Disjunctive Datalog Rules
(DDR)~\cite{DBLP:journals/theoretics/KhamisNS25}.  More precisely,
consider the following DDR:
\begin{align*}
  U_1(X_1,X_2,X_3)& \vee U_2(X_1,X_2,X_4)\vee U_3(X_3,X_4) \vee
                    U_4(X_1,X_5) \vee U_5(X_2,X_5) \\
\leftarrow\ & A(X_1,X_2) \wedge B(X_1,X_3) \wedge C(X_1,X_4) \wedge D(X_2,X_3)\\
             & {}\wedge E(X_2,X_4) \wedge F(X_3,X_5) \wedge G(X_4,X_5)
\end{align*}
The entropic- or polymatroid-bound of this DDR is the quantity
$\max_h M(h)$, where $h$ is an almost entropic polymatroid, or an
arbitrary polymatroid respectively ($h \in \bar\Gamma_X^*$ or
$h \in \Gamma_X$).  Denote these two bounds by $e$ and $p$
respectively: thus $e \leq 38/29$ and $p=4/3$.  Then, the results
in~\cite{DBLP:journals/theoretics/KhamisNS25} prove that, on one hand,
for any input relations $A, B, \ldots, G$ of size $N$, there exists an
output of the DDR rule where $\max(|U_1|,|U_2|,|U_3|,|U_4|,|U_5|)$ is
$\leq N^e$, and, on the other hand, there exists an algorithm that can
compute an output of size $O(N^p)$ in time $\tilde O(N^p)$.
Theorem~\ref{thm:central} proves that there is a gap between these two
bounds.

\subsection{Modular Polymatroids}

\begin{lemma}\label{lem:modular-core}
Every modular polymatroid $h\in M_X$ satisfies:
\begin{equation}\label{eq:modular-core}
  h(U_3)+h(U_4)+h(U_5)
  =h(\xs{12})+h(\xs{35})+h(\xs{45}).
\end{equation}
In particular, if $h$ is an edge dominated modular polymatroid on the
central graph $G_0$, then $\min_i h(U_i)\le1$.
This bound is attained by $h(U)=|U|/2$.
\end{lemma}

\begin{proof}
By modularity, both sides of the identity equal
\[
  h(X_1)+h(X_2)+h(X_3)+h(X_4)+2h(X_5).
\]
If $h$ is edge dominated, then
\[
  3\min_i h(U_i)
  \le h(U_3)+h(U_4)+h(U_5)
  =h(\xs{12})+h(\xs{35})+h(\xs{45})
  \le3,
\]
since $\xs{12},\xs{35},\xs{45}$ are central edges.
For $h(U)=|U|/2$, every central edge has value $1$, and
$\min_i h(U_i)=1$.  This proves the lemma.
\end{proof}

\subsection{Normal Polymatroids}

\begin{lemma}\label{lem:normal-core}
Every normal polymatroid $h \in N_X$ satisfies:
\begin{equation}\label{eq:normal-core}
\begin{aligned}
&h(\xs{123})+h(\xs{124})+5h(\xs{34})+2h(\xs{15})+2h(\xs{25})\\
&\qquad\le2h(\xs{13})+2h(\xs{14})+2h(\xs{23})+2h(\xs{24})
        +3h(\xs{35})+3h(\xs{45}).
\end{aligned}
\end{equation}
In particular, if $h$ is an edge dominated normal polymatroid on the
central graph $G_0$ then $\min_i h(U_i)\le14/11$.
\end{lemma}
\begin{proof}
To prove
inequality~\eqref{eq:normal-core} we use the fact that every normal
polymatroid is a nonnegative linear combination of step functions,
$h = \sum_{K \subseteq X}m(K)\,h^K$ with all $m(K)\ge0$, where
$h^K(U)=\bm 1_{U\cap K\neq\emptyset}$ (see Section~\ref{sec:basic}).
Therefore, it suffices to check inequality~\eqref{eq:normal-core} for
a step function $h^K$:
\begin{align*}
  & \mathbf1_{\{K\cap\xs{123}\ne\varnothing\}}
    +\mathbf1_{\{K\cap\xs{124}\ne\varnothing\}}
    +5\cdot\mathbf1_{\{K\cap\xs{34}\ne\varnothing\}}
    +2\cdot\mathbf1_{\{K\cap\xs{15}\ne\varnothing\}}
    +2\cdot\mathbf1_{\{K\cap\xs{25}\ne\varnothing\}}\leq\\
  & \quad 2\cdot\mathbf1_{\{K\cap\xs{13}\ne\varnothing\}}
    +2\cdot\mathbf1_{\{K\cap\xs{14}\ne\varnothing\}}
    +2\cdot\mathbf1_{\{K\cap\xs{23}\ne\varnothing\}}
    +2\cdot\mathbf1_{\{K\cap\xs{24}\ne\varnothing\}}\\
  & \quad {}+3\cdot\mathbf1_{\{K\cap\xs{35}\ne\varnothing\}}
    +3\cdot\mathbf1_{\{K\cap\xs{45}\ne\varnothing\}}
\end{align*}
This inequality can be checked exhaustively for all 32 possible sets
$K \subseteq \set{X_1, \ldots, X_5}$.
The same exhaustive check shows that equality holds iff $K$ is one of
the following sets:
\begin{equation}\label{eq:tight-atoms}
X_1,\quad X_2,\quad \xs{13},\quad \xs{23},\quad \xs{123},\quad \xs{14},\quad \xs{24},\quad \xs{124},\quad \xs{35},\quad \xs{45}.
\end{equation}

The inequality implies for edge dominated normal polymatroids:
$11\min_i h(U_i) \leq
h(\xs{123})+h(\xs{124})+5h(\xs{34})+2h(\xs{15})+2h(\xs{25})
\le2h(\xs{13})+2h(\xs{14})+2h(\xs{23})+2h(\xs{24})
+3h(\xs{35})+3h(\xs{45}) \leq 2+2+2+2+3+3=14$.
\end{proof}

Inequality \eqref{eq:normal-core} is genuinely a normal inequality: it
fails for the polymatroid $g_P$ of Lemma~\ref{lem:subw-tight} below,
whose left-hand side is $11\cdot\frac43>14$, and, by
Lemma~\ref{lem:ingleton-tight} below, it already fails for a linear
polymatroid, so it is not a linear rank inequality.

\begin{lemma}\label{lem:normal-tight}
  There exists an edge dominated, normal polymatroid $g_N$ on the
  central graph $G_0$ such that $g_N(U_i)=14/11$ for $i=1,\ldots,5$.
  Furthermore,
  \begin{equation}\label{eq:normal-singletons}
    g_N(X_1)=g_N(X_2)=\frac6{11}<\frac23,\qquad
    g_N(X_3)=g_N(X_4)=\frac7{11}<\frac23,\qquad
    g_N(X_5)=\frac8{11}<\frac34.
  \end{equation}
\end{lemma}

\begin{proof}
  We define $g_N$ to be the following nonnegative combination of the
  ten step functions $h^K$ given by the sets $K$
  in~\eqref{eq:tight-atoms}; the ten coefficients sum to $18/11$:
\begin{equation}\label{eq:normal-witness}
\begin{aligned}
  g_N = \frac{1}{11}\bigl(
  &2h^{\xs{1}}+2h^{\xs{2}}
   +h^{\xs{13}}+h^{\xs{23}}+h^{\xs{123}}\\
  &{}+h^{\xs{14}}+h^{\xs{24}}+h^{\xs{124}}
   +4h^{\xs{35}}+4h^{\xs{45}}\bigr)
\end{aligned}
\end{equation}
Observe that $g_N(\xs{12})=10/11$ and $g_N(e)=1$ for the other six
central edges $e\in E_0$, so $g_N$ is edge dominated. We also observe
that $g_N(U_i)=14/11$ for $i=1,\ldots,5$. The singleton values follow
directly from \eqref{eq:normal-witness}. This proves the lemma.
\end{proof}

\subsection{Linear Polymatroids}

\label{subsec:shannon-core}

The remaining three bounds, for linear, almost entropic and general
polymatroids, all come from the same source.  We start from the identity:
\begin{equation}\label{eq:shannon-core}
\begin{aligned}
h(\xs{123})+h(\xs{124})+h(\xs{34})
&+\bigl[I_h(X_1;X_2\mid X_3)+I_h(X_1;X_2\mid X_4)+I_h(X_3;X_4)\bigr]\\
&= h(\xs{13})+h(\xs{23})+h(\xs{14})+h(\xs{24}),
\end{aligned}
\end{equation}
which can be checked immediately by expanding the three information
terms.  Dropping the bracket gives the polymatroid bound of
Lemma~\ref{lem:subw-core}; replacing it by the sharper lower bound
$I_h(X_1;X_2)$ supplied by \eqref{eq:ing} gives
Lemma~\ref{lem:ingleton}; and \eqref{eq:zy} gives the intermediate
Lemma~\ref{lem:zy-core}.

\begin{lemma}\label{lem:ingleton}
  Every linear polymatroid $h \in \text{Lin}_X$ satisfies
\begin{align}
h(\xs{123})+h(\xs{124})+3h(\xs{34})+h(\xs{15})+h(\xs{25})
&\le h(\xs{12})+h(\xs{13})+h(\xs{14})+h(\xs{23})+h(\xs{24})\notag\\
&\quad+2h(\xs{35})+2h(\xs{45}).
\label{eq:ing-core}
\end{align}
In particular, if $h$ is an edge dominated linear polymatroid on the
central graph $G_0$, then $\min_i h(U_i)\le9/7$.
\end{lemma}
\begin{proof}
The inequality follows immediately by adding up the following
inequalities:
\begin{align*}
  & 1 \times & h(\xs{123})+&h(\xs{124})+h(\xs{34})+h(X_1)+h(X_2)\\
  &          & \le\ & h(\xs{12})+h(\xs{13})+h(\xs{14})+h(\xs{23})+h(\xs{24})
  && \text{from \eqref{eq:shannon-core} and \eqref{eq:ing}}\\
  & 1 \times & h(\xs{15})\le\ & h(X_1)+h(X_5)
  && \text{from \eqref{eq:normalized}, \eqref{eq:submodular}}\\
  & 1 \times & h(\xs{25})\le\ & h(X_2)+h(X_5)
  && \text{from \eqref{eq:normalized}, \eqref{eq:submodular}}\\
  & 2 \times & h(\xs{345})+h(X_5)\le\ & h(\xs{35})+h(\xs{45})
  && \text{from \eqref{eq:submodular}}\\
  & 2 \times & h(\xs{34})\le\ & h(\xs{345})
  && \text{from \eqref{eq:monotone}}
\end{align*}
Each inequality needs to be multiplied by the factor on the left, then
added up.  The terms $h(X_1)$, $h(X_2)$, $2h(X_5)$ and $2h(\xs{345})$
occur on both sides and cancel, and what remains is
inequality~\eqref{eq:ing-core}.  The first line above uses Ingleton's
inequality~\eqref{eq:ing}, which is valid because $h$ is a linear
polymatroid.  All the others are basic Shannon inequalities, valid for
all polymatroids.  Our discussion also implies that the right-hand
side of \eqref{eq:ing-core} minus its left-hand side is
\begin{equation}\label{eq:ing-cert}
\Ing_h(X_1,X_2;X_3,X_4)+I_h(X_1;X_5)+I_h(X_2;X_5)+2I_h(X_3;X_4\mid X_5)+2H_h(X_5\mid\xs{34}).
\end{equation}

For the consequence, assume $h(e)\le1$ for every $e\in E_0$.  The
LHS of~\eqref{eq:ing-core} is
$h(U_1)+h(U_2)+3h(U_3)+h(U_4)+h(U_5)\ge7\min_i h(U_i)$, while the RHS
is  $\leq 1+1+1+1+1+2+2=9$.
\end{proof}

The bound $9/7$ is attained, by an explicit configuration of subspaces.

\begin{lemma}\label{lem:ingleton-tight}
  There exists an edge dominated, linear polymatroid
  $g_L \in \text{Lin}_X$ on the central graph $G_0$ such that
  $g_L(U_i)=9/7$ for $i=1,\ldots,5$.

Furthermore, the polymatroid $g_L$ satisfies:
\begin{equation}\label{eq:linear-singletons}
g_L(X_1)=g_L(X_2)=\frac47<\frac23,\qquad
g_L(X_3)=g_L(X_4)=\frac9{14}<\frac23,\qquad
g_L(X_5)=\frac57<\frac34 .
\end{equation}
\end{lemma}

\begin{proof}
  Consider the vector space $W=\R^{18}$ with basis
  $a_1,\ldots,a_9,b_1,\ldots,b_9$.  We define $g_L := \frac{1}{14}r$,
  where $r$ is the rank function of the following coverage
  function:\footnote{More formally, the coverage function $\varphi$ is
    $\varphi(X_3):=\set{a_1, \ldots, a_9}$,
    $\varphi(X_4):=\set{b_1, \ldots, b_9}$, etc: what we show in
    Eq.~\eqref{eq:linear-witness} is directly
    $V_i := \Span(\varphi(X_i))$.}
\begin{equation}\label{eq:linear-witness}
\begin{aligned}
V_3&=\langle a_1,\ldots,a_9\rangle,\qquad
V_4=\langle b_1,\ldots,b_9\rangle,\qquad
V_5=\langle a_1,\ldots,a_5,b_1,\ldots,b_5\rangle,\\
V_1&=\langle a_6,\ a_7,\ a_8,\ b_6,\ b_7,\ b_8,\
   a_9+a_2+b_3,\ b_9+a_3+b_2\rangle,\\
V_2&=\langle a_9,\ a_7+a_1,\ a_8+a_1,\ b_9,\ b_7+b_1,\ b_8+b_1, a_6+a_4+b_5,\ b_6+a_5+b_4\rangle.
\end{aligned}
\end{equation}
As in Section~\ref{sec:basic}, the rank function is
$r(S):=\dim\bigl(\sum_{X_i\in S}V_i\bigr)$.

Observe that $V_1,\ldots,V_5$ have dimensions $8$, $8$, $9$, $9$, $10$
respectively, proving the bounds~\eqref{eq:linear-singletons}.

  For the remaining values, note first that $W=V_3+ V_4$, so
  $r(\xs{34})=18$; and $V_5=(V_5\cap V_3)+(V_5\cap V_4)$ with
  both summands of dimension $5$, so
  $r(\xs{35})=r(\xs{45})=9+10-5=14$.  Reducing modulo $V_5$ sends the
  eight displayed generators of $V_1$ to
  $\bar a_6,\bar a_7,\bar a_8,\bar b_6,\bar b_7,\bar b_8,\bar a_9,\bar
  b_9$, a basis of $W/V_5$, so $V_1\cap V_5=0$ and
  $r(\xs{15})=8+10=18$; the generators of $V_2$ reduce to the same
  basis in a different order, so $r(\xs{25})=18$ likewise.  Next
\[
V_1\cap V_3=\langle a_6,a_7,a_8\rangle,\qquad
V_2\cap V_3=\langle a_9,\ a_1+a_7,\ a_1+a_8\rangle,
\]
both of dimension $3$, with the mirror statements for $V_4$ obtained by
exchanging $a_i\leftrightarrow b_i$; hence
$r(\xs{13})=r(\xs{14})=r(\xs{23})=r(\xs{24})=8+9-3=14$.  Also
$V_1\cap V_2=\langle a_7-a_8,\ b_7-b_8\rangle$ has dimension $2$, so
$r(\xs{12})=8+8-2=14$ and every central edge has $g_L$-value $1$.
Finally, let $\pi_4:W\to V_4$ be the projection defined by
$\pi_4(a_i)=0$ and $\pi_4(b_i)=b_i$. The projections of the
generators of $V_1$ and $V_2$ span $V_4$: they include
$b_3,b_5,b_6,b_7,b_8,b_9$, and the remaining basis vectors are
obtained from
\[
b_1=(b_7+b_1)-b_7,\qquad
b_2=(b_9+b_2)-b_9,\qquad
b_4=(b_6+b_4)-b_6.
\]
Since $\ker\pi_4=V_3$, it follows that $V_1+V_2+V_3=W$.
Exchanging $a_i$ and $b_i$ gives $V_1+V_2+V_4=W$.
Hence $r(\xs{123})=r(\xs{124})=18$, so
$r(U_i)=18$ and $g_L(U_i)=18/14=9/7$ for all five $i$.
\end{proof}

Since $9/7>14/11$, the configuration \eqref{eq:linear-witness} is
linear but not normal; we return to this in
Section~\ref{sec:linear-upper}.

The polymatroid $g_L$ is not a normal polymatroid.  This can be seen
by observing that its restriction to the set $\set{X_1,X_2,X_5}$
violates \eqref{eq:normal:polymatroids} at $K=\set{X_1,X_2,X_5}$.  We
check this inequality for the rank function:
\[
\begin{aligned}
&r(X_1)+r(X_2)+r(X_5)-r(\xs{12})-r(\xs{15})-r(\xs{25})+r(\xs{125})\\
&\qquad{}=8+8+10-14-18-18+18=-6<0,
\end{aligned}
\]
using $r(\xs{125})=18$, which follows from $r(\xs{15})=18=\dim W$ by
monotonicity.  This is the same computation that gave $-1$ for the
linear polymatroid in Section~\ref{sec:three:special}.

\subsection{Almost Entropic Polymatroids}

\label{subsec:zy-core}

\begin{lemma}\label{lem:zy-core}
  Every almost entropic polymatroid $h \in \bar \Gamma_X^*$ satisfies
\begin{align}
5h(\xs{123})&+5h(\xs{124})+11h(\xs{34})+4h(\xs{15})+4h(\xs{25}) \notag\\
&\le 6h(\xs{12})+4h(\xs{13})+4h(\xs{14})+4h(\xs{23})+4h(\xs{24})+8h(\xs{35})+8h(\xs{45}).
\label{eq:zy-core}
\end{align}
In particular, if $h$ is an edge dominated, almost entropic polymatroid on
the central graph $G_0$, then $\min_i h(U_i)\le38/29$.
\end{lemma}

\begin{proof}
Apply \eqref{eq:zy} twice, with $(X,Y,A,B)$ equal to
$(X_1,X_2,X_3,X_4)$ and to $(X_1,X_2,X_4,X_3)$; the two instances play
the symmetric roles that the single instance of \eqref{eq:ing} played
above, and their sum replaces $\Ing_h(X_1,X_2;X_3,X_4)$.  The
inequality follows by adding up the following inequalities:
\begin{align*}
  & 1 \times & & 4h(\xs{123})+h(\xs{124})+h(\xs{34})+2h(X_1)+2h(X_2)+h(X_3)
  && \text{from \eqref{eq:zy}}\\
  &          & \le\ & 3h(\xs{12})+3h(\xs{13})+3h(\xs{23})+h(\xs{14})+h(\xs{24})\\
  & 1 \times & & h(\xs{123})+4h(\xs{124})+h(\xs{34})+2h(X_1)+2h(X_2)+h(X_4)
  && \text{from \eqref{eq:zy}}\\
  &          & \le\ & 3h(\xs{12})+3h(\xs{14})+3h(\xs{24})+h(\xs{13})+h(\xs{23})\\
  & 1 \times & h(\xs{34})\le\ & h(X_3)+h(X_4)
  && \text{from \eqref{eq:normalized}, \eqref{eq:submodular}}\\
  & 4 \times & h(\xs{15})\le\ & h(X_1)+h(X_5)
  && \text{from \eqref{eq:normalized}, \eqref{eq:submodular}}\\
  & 4 \times & h(\xs{25})\le\ & h(X_2)+h(X_5)
  && \text{from \eqref{eq:normalized}, \eqref{eq:submodular}}\\
  & 8 \times & h(\xs{345})+h(X_5)\le\ & h(\xs{35})+h(\xs{45})
  && \text{from \eqref{eq:submodular}}\\
  & 8 \times & h(\xs{34})\le\ & h(\xs{345})
  && \text{from \eqref{eq:monotone}}
\end{align*}
Each inequality needs to be multiplied by the factor on the left, then
added up.  The terms $4h(X_1)$, $4h(X_2)$, $h(X_3)$, $h(X_4)$,
$8h(X_5)$ and $8h(\xs{345})$ occur on both sides and cancel, and what
remains is inequality~\eqref{eq:zy-core}.  Equivalently, the right-hand
side of \eqref{eq:zy-core} minus its left-hand side is
\begin{equation}\label{eq:zy-cert}
\begin{aligned}
&\ZY_h(X_1,X_2;X_3,X_4)+\ZY_h(X_1,X_2;X_4,X_3)+I_h(X_3;X_4)\\
&\qquad+4I_h(X_1;X_5)+4I_h(X_2;X_5)+8I_h(X_3;X_4\mid X_5)+8H_h(X_5\mid\xs{34}),
\end{aligned}
\end{equation}
the sum of the slacks of the seven inequalities above, with the same
multiplicities.  Only the first two use \eqref{eq:zy}; the other five
are Shannon inequalities.

Note the parallel with \eqref{eq:ing-cert}: the four Shannon terms
$H_h(X_5\mid\xs{34})$, $I_h(X_1;X_5)$, $I_h(X_2;X_5)$, $I_h(X_3;X_4\mid X_5)$ occur with
exactly four times their earlier weights, and the extra $I_h(X_3;X_4)$
together with the replacement of $\Ing$ by $\ZY$ is what forces the
coefficients $5,5,11$ and $6$ in place of $4,4,12$ and $4$.

For the consequence, assume $h(e)\le1$ for every $e\in E_0$.  The
coefficients on the left of \eqref{eq:zy-core} sum to $29$, so its
left-hand side is at least $29\min_i h(U_i)$; the coefficients on the
right sum to $38$, and every set there is a central edge, so its
right-hand side is at most $38$.
\end{proof}

The upper bound $38/29$ is not known to be attained.  The next lemma
gives a lower bound of $130/101$.

\begin{lemma}\label{lem:zy-tight}  Define the number 
  $a:=H_2(1/4)=2-\frac34\log_2 3$.  There exists an edge dominated,
  almost entropic polymatroid $g_E$ on the central graph $G_0$ such that
\begin{align}
g_E(e)&=1 \text{ for all edges } e\in E_0,\label{eq:core-edges}\\
g_E(U_1)=g_E(U_2)=g_E(U_3)&=\frac{130}{101},\qquad
g_E(U_4)=g_E(U_5)=\frac{122+10a}{101}>\frac{130}{101}.
\label{eq:core-large}
\end{align}

Furthermore, $g_E$ satisfies:
\begin{align}
g_E(X_1)=g_E(X_2)&=\frac{50+10a}{101}<\frac23,\qquad
g_E(X_3)=g_E(X_4)=\frac{65}{101}<\frac23,\notag\\
g_E(X_5)&=\frac{72}{101}<\frac34.
\label{eq:core-singletons}
\end{align}
\end{lemma}
\begin{proof}
  We will use the following inequalities below: $4/5<a<1$ and
  $3^5<2^8$.

  Let $R,S$ be independent fair bits, and let $h_0$ be the entropy
  function, in bits, of
\[
(X_1 := R\wedge S,\ X_2 := R\vee S,\ X_3 := R,\ X_4 := S,\ X_5 := 0).
\]
More precisely, $h_0$ is the entropy of the following probability
distribution:
\begin{align*}
  &
    \begin{array}{|cc||c|c|c|c|c|c} \cline{1-7}
R & S &  X_1 := R\wedge S&X_2:=R\vee S&X_3:=R&X_4:=S&X_5:=0&p \\ \cline{1-7}
0 & 0 & 0 & 0 & 0 & 0 & 0 & 1/4 \\
0 & 1 & 0 & 1 & 0 & 1 & 0 & 1/4 \\
1 & 0 & 0 & 1 & 1 & 0 & 0 & 1/4 \\
1 & 1 & 1 & 1 & 1 & 1 & 0 & 1/4 \\ \cline{1-7}
    \end{array}
\end{align*}
In particular,
$h_0(X_1)=h_0(X_2) = \frac{1}{4}\log_2 4 + \frac{3}{4}\log_2 (4/3) =
2-\frac{3}{4}\log_2 3 = a$, $h_0(X_3)=h_0(X_4)=1$ and $h_0(X_5)=0$.
Notice that the restriction of $h_0$ to $X_1, \ldots, X_4$ is
precisely the entropic, non-linear polymatroid defined in
Section~\ref{sec:three:special}.

Similarly, let $q_1,\ldots,q_4$ be the entropy functions of the
following five-tuples; the fair bits used in different independent
copies are always fresh ($R\oplus S$ denotes xor):
\[
\begin{array}{c|ccccc}
 &X_1&X_2&X_3&X_4&X_5\\\hline
q_1&0&0&0&R&R\\
q_2&0&R&0&S&S\\
q_3&R&R&0&(R,S)&S\\
q_4&R&S&0&(R,S)&R\oplus S
\end{array}
\]
Each $q_j$ is a linear polymatroid.  Let $\Pi$ consist of the
following four permutations of $\{1,\ldots,5\}$:
\[
\Pi=\{\mathrm{id},(12),(34),(12)(34)\}.
\]
For $\pi\in\Pi$ and $B\subseteq X$, write
$\pi(B):=\{X_{\pi(i)}:X_i\in B\}$.
Thus $q_j(\pi(B))$ evaluates $q_j$ on the set obtained by
relabeling the variables in $B$ according to $\pi$.
For example, if $\pi=(12)$ and $B=\xs{135}$, then
$\pi(B)=\xs{235}$. Define
\[
s_j(B)=\frac14\sum_{\pi\in\Pi}q_j(\pi(B)),\qquad
g_E=\frac{10h_0+10s_1+24s_2+14s_3+24s_4}{101}.
\]
Each $\pi\in\Pi$ maps central edges to central edges and permutes
the sets $U_1,\ldots,U_5$. For example, $(12)$ exchanges $U_4$
and $U_5$, while $(34)$ exchanges $U_1$ and $U_2$.
Consequently, averaging over these permutations preserves edge
domination and any common lower bound on the values at the $U_i$.
Since each $q_j\circ\pi$ and $h_0$ is an entropy function, $g_E$ is a
nonnegative combination of entropy vectors and hence almost entropic.
More explicitly,
\begin{equation}\label{eq:finite-entropy}
202g_E=20h_0+\sum_{\pi\in\Pi}
\bigl(5q_1\circ\pi+12q_2\circ\pi+7q_3\circ\pi+12q_4\circ\pi\bigr)
\end{equation}
is an actual finite entropy vector: concatenate the indicated integer
numbers of independent copies at each coordinate.

For transparency, all values claimed above are obtained from
the following entropy table:
\[
\begin{array}{c|ccccc}
B&h_0(B)&s_1(B)&s_2(B)&s_3(B)&s_4(B)\\\hline
X_1,X_2&a&0&1/2&1&1\\
X_3,X_4&1&1/2&1/2&1&1\\
X_5&0&1&1&1&1\\
\xs{12}&3/2&0&1&1&2\\
\xs{13},\xs{14},\xs{23},\xs{24}&3/2&1/2&1&3/2&3/2\\
\xs{35},\xs{45}&1&1&1&3/2&3/2\\
\xs{123},\xs{124}&2&1/2&3/2&3/2&2\\
\xs{34}&2&1&1&2&2\\
\xs{15},\xs{25}&a&1&3/2&2&2
\end{array}
\]
Substitution gives \eqref{eq:core-edges}--\eqref{eq:core-singletons};
every strict inequality there uses only $4/5<a<1$.
\end{proof}

This explicitly separates the entropic optimization on the central graph from its
Ingleton upper bound, since $130/101>9/7$; consistently with
Lemma~\ref{lem:zy-core}, $130/101<38/29$.
For the upper bound $38/29$ we verified with a linear programming
solver that it is tight for the relaxation from which we derived it:
$38/29$ is the exact value of $\max_h M(h)$ over the edge dominated
polymatroids on $G_0$ that satisfy every instance of the Zhang--Yeung
inequality~\eqref{eq:zy}.  That relaxation is a polyhedron containing
$\bar \Gamma_X^* \cap ED(G_0)$, so for $\Gamma_X^*$ or
$\bar \Gamma_X^*$ the value $38/29$ is only an upper bound, and we do
not know whether it is tight.  Nor can a linear programming solver
settle the question, because $\bar \Gamma_X^*$ is not a polyhedron:
$\max_{h \in \bar \Gamma_X^* \cap ED(G_0)} M(h)$ is not the optimum of
any linear program.

\subsection{General Polymatroids}

\label{subsec:subw-core}

\begin{lemma}\label{lem:subw-core}
  Every polymatroid $h \in \Gamma_X$ satisfies
\begin{align}
h(\xs{123})+h(\xs{124})+h(\xs{34})
&\le h(\xs{13})+h(\xs{23})+h(\xs{14})+h(\xs{24}).
\label{eq:subw-core}
\end{align}
In particular, if $h$ is an edge dominated polymatroid on the central
graph $G_0$ then $\min_i h(U_i)\le4/3$.
\end{lemma}
\begin{proof}
The inequality follows immediately by adding up the following
inequalities:
\begin{align*}
  & 1 \times & h(\xs{123})+h(X_3)\le\ & h(\xs{13})+h(\xs{23})
  && \text{from \eqref{eq:submodular}}\\
  & 1 \times & h(\xs{124})+h(X_4)\le\ & h(\xs{14})+h(\xs{24})
  && \text{from \eqref{eq:submodular}}\\
  & 1 \times & h(\xs{34})\le\ & h(X_3)+h(X_4)
  && \text{from \eqref{eq:normalized}, \eqref{eq:submodular}}
\end{align*}

For the consequence, assume $h(e)\le1$ for every $e\in E_0$.  The LHS
of~\eqref{eq:subw-core} is $h(U_1)+h(U_2)+h(U_3)\ge3\min_i h(U_i)$,
while the RHS is $\le 1+1+1+1=4$ for any edge dominated polymatroid.
\end{proof}

The bound $4/3$ is attained, by an explicit polymatroid.

\begin{lemma}\label{lem:subw-tight}
  There exists an edge dominated polymatroid $g_P \in \Gamma_X$ on the
  central graph $G_0$ such that $g_P(U_i)=4/3$ for $i=1,\ldots,5$.

Furthermore, the polymatroid $g_P$ satisfies $g_P(X_i)=2/3$ for
$i=1,\ldots,5$, and $g_P(B)=4/3$ for every $B\subseteq X$ that is contained in
no central edge.
\end{lemma}

\begin{proof}
  Recall from Eq.~\eqref{eq:function:f:g} in
  Section~\ref{sec:three:special} the polymatroid induced by a graph,
  and apply it to the central graph: set
  $g_P(U) \defeq \frac{1}{3}f_{G_0}(U)$ for all $U \subseteq X$.

  Observe that $g_P(e)=1$ for every edge $e \in E_0$, so $g_P$ is edge
  dominated, and that $g_P(X_i)=2/3$.  Since every central edge has
  two elements, the sets contained in some central edge are exactly
  $\varnothing$, the five singletons and the seven elements of $E_0$,
  so $g_P(B)=4/3$ precisely when $B$ is contained in no central edge.
  By \eqref{eq:the:sets:u} no $U_i$ is contained in a central edge.
  We also observe that $g_P(U_i)=4/3$ for $i=1,\ldots,5$.  This proves the
  lemma.
\end{proof}

\section{The Big Graph}
\label{sec:big-graph}

We now use the central bounds of Theorem~\ref{thm:central} to
prove the width separations on the big graph $G$. The scaffolding
introduced in Section~\ref{sec:graph} connects the values
$M(h)=\min_i h(U_i)$ to the values of bags in tree decompositions.

For the upper bounds, Lemma~\ref{lem:decomposition} constructs,
for every edge dominated polymatroid $h$, a tree decomposition
showing that
\[
\hwd_h(G)\le\max\{6,5+M(h)\}.
\]
Thus each central upper bound gives an upper bound on the
corresponding width of $G$.

For the lower bounds, Lemma~\ref{lem:bramble} identifies a bag
in every tree decomposition that contains some $U_i$ and meets
every edge joining an anchor to a connector.
For adaptive width, this forces a bag with at least $12$ vertices,
giving the lower bound $6$ under the uniform modular polymatroid
$h(B)=|B|/2$. For the other four widths, we extend each central
witness $g_0$ to an edge dominated polymatroid $F_{g_0}$ of the
same class on $G$. The extension ensures that the scaffolding
contributes $5$ to the value of the selected bag, so
Lemma~\ref{lem:witness-bound} gives
\[
\hwd_{F_{g_0}}(G)\ge5+M(g_0).
\]
Together, these bounds yield the separations stated in
Theorem~\ref{thm:main}.

Throughout this section we denote by $N[x]$ the \emph{closed}
neighborhood of a node $x$, that is, $\set{x}$ together with all
neighbors of $x$.
For every $a\in A_i$ and $1\le k\le r_i$, set
$K_{a,k}=\{a,p_{i,k}\}$.

\subsection{The Decomposition Lemma}

Recall from Section~\ref{sec:core-bounds} that
$M(h) = \min(h(U_1), h(U_2), h(U_3), h(U_4), h(U_5))$
(Eq.~\eqref{eq:mh}).

\begin{lemma}\label{lem:decomposition}
For every edge-dominated polymatroid $h$ on $G$,
\[
  \hwd_h(G)\le\max(6,\ 5+M(h)).
\]
\end{lemma}

By Theorem~\ref{thm:central}, $\max_h M(h)\in [1,4/3]$ when $h$
ranges over any of the classes of polymatroids considered in this
paper.  Therefore, the dominant term for $\hwd_h$ is $5+M(h)$, a
quantity coming from the central graph $G_0$, for which we have
precise bounds from Theorem~\ref{thm:central}.  The contribution of
the rest of the big graph $G$ to $\hwd_h$ is at most $6$ and we will
ignore it when we compute the upper bounds of the various notions of
width.

\begin{proof} [Proof of Lemma~\ref{lem:decomposition}]
  It suffices to describe one tree decomposition $T$ for which
  $\max_{t\in V(T)}h(B_t) \leq \max(6,\ 5+M(h))$.  We show $T$ in
  Figure~\ref{fig:td-generic} and describe it here.  We will refer to
  the components of the big graph $G$ as defined in
  Section~\ref{sec:graph}: anchor nodes $A$, and connector nodes $p$.
  Choose $i\in\set{1,\ldots,5}$ minimizing $h(U_i)$, and take $j=2$ if $i=1$,
  otherwise $j=1$. Construct the following three  bags, call them
  \emph{main bags}, and connect them in a chain:
\[
C_1=A\cup U_i,\qquad
C_2=(A\setminus A_i)\cup(U_i\cup U_j),\qquad
C_3=(A\setminus(A_i\cup A_j))\cup X.
\]
For each connector $p_{i,\ell}$, $\ell=1,\ldots,r_i$, create a bag
$N[p_{i,\ell}]$ consisting of its closed neighborhood, and connect it
to $C_1$.  Similarly, for each connector in group $j$ create a bag
$N[p_{j,\ell}]$ and attach it to $C_2$. For all remaining connector
nodes $p$, create a bag $N[p]$ and attach it to $C_3$.  The resulting
tree is sketched in Figure~\ref{fig:td-generic}.  The complete tree
for the case $i=1$ is shown in Figure~\ref{fig:td-concrete}: it has
$20$ bags (ignore the function $F_E$ for
now).  This is a correct tree decomposition, because all neighbors of
$p_{k,\ell}$ are included in $A_k \cup U_k$, which, by
construction, belongs to the appropriate main bag $C_1, C_2$, or $C_3$
respectively: $A_i\cup U_i\subseteq C_1$, $A_j\cup U_j\subseteq C_2$,
and $A_k\cup U_k\subseteq C_3$ for $k\notin\{i,j\}$.  The first main
bag covers all anchor edges and the last covers all central
edges. Along the main path central vertices are only added and anchors
only removed, so all vertex occurrences are connected.

Recall that the anchor nodes consist of five sets:
$A = A_1 \cup \cdots \cup A_5$, and any two of them, $A_k$ and $A_{k'}$, induce
a complete bipartite graph, $K_{2,2}$, with $4$ edges.  Consider the
union of any $2 \le m\le 5$ of the anchor groups.  Their induced
subgraph has a fractional edge cover of weight $m$, because there are
$2m$ vertices, and each vertex has exactly $2(m-1)$ neighbors.  If we
give every edge a weight of $1/[2(m-1)]$, then every vertex is
covered.  There are $\binom{m}{2}\cdot 4 = 2m(m-1)$ edges in this
induced subgraph, so the total weight is
$2m(m-1)\cdot\frac{1}{2(m-1)}=m$.  By \eqref{eq:cover} with $m=5,4,3$,
for any edge dominated polymatroid $h$, the anchor parts of
$C_1, C_2, C_3$ have values at most $5,4,3$, more precisely
$h(A) \leq 5$, $h(A\setminus A_i)\leq 4$ and
$h(A\setminus(A_i \cup A_j))\leq 3$.

For $i=1,2,3$, the set $U_i\cup U_j=\xs{1234}$ is covered by $\xs{13},\xs{24}$;
for $i=4,5$, the set $U_i\cup U_j=\xs{1235}$ is covered by $\xs{12},\xs{35}$.
The whole central graph is covered by $\xs{12},\xs{35},\xs{45}$. Consequently
\[
h(C_1)\le5+h(U_i),\qquad h(C_2)\le6,\qquad h(C_3)\le6.
\]

It remains to compute $h(N[p])$ for every connector node $p$.  Each
connector has degree at most five, and its incident edges cover all
its neighbors, $N[p]$. Thus every connector bag $N[p]$ has
$h(N[p]) \leq 5$.  This completes the proof of
Lemma~\ref{lem:decomposition}.
\end{proof}

\begin{figure}[t]
\centering
\begin{tikzpicture}[y=1cm,x=1cm]
\node[mainbag] (C1) at (0,0)
  {$C_1=A\cup U_i$\\[1pt]
   \textcolor{black!70}{\scriptsize anchors $\le5$, central part $U_i$}\\[1pt]
   $h(C_1)\le5+h(U_i)$};
\node[mainbag] (C2) at (0,-2.45)
  {$C_2=(A\setminus A_i)\cup(U_i\cup U_j)$\\[1pt]
   \textcolor{black!70}{\scriptsize anchors $\le4$, central part in two central edges}\\[1pt]
   $h(C_2)\le4+2=6$};
\node[mainbag] (C3) at (0,-4.9)
  {$C_3=\bigl(A\setminus(A_i\cup A_j)\bigr)\cup X$\\[1pt]
   \textcolor{black!70}{\scriptsize anchors $\le3$, central part in three central edges}\\[1pt]
   $h(C_3)\le3+3=6$};
\draw[tlink] (C1) -- (C2);
\draw[tlink] (C2) -- (C3);
\def\lx{6.3}
\node[leafbag] (a1) at (\lx, 0.52) {$N[p_{i,1}]$};
\node          (a2) at (\lx, 0.02) {$\vdots$};
\node[leafbag] (a3) at (\lx,-0.52) {$N[p_{i,r_i}]$};
\node[leafbag] (b1) at (\lx,-1.93) {$N[p_{j,1}]$};
\node          (b2) at (\lx,-2.43) {$\vdots$};
\node[leafbag] (b3) at (\lx,-2.97) {$N[p_{j,r_j}]$};
\node[leafbag] (c1) at (\lx,-4.38) {$N[p_{k,1}]$};
\node          (c2) at (\lx,-4.88) {$\vdots$};
\node[leafbag] (c3) at (\lx,-5.42) {$N[p_{k,r_k}]$};
\foreach \s/\t in {C1/a1,C1/a3,C2/b1,C2/b3,C3/c1,C3/c3}
  \draw[tlink] (\s.east) -- (\t.west);
\node[bnote,anchor=west] at (7.45, 0.00) {the $r_i$ connectors\\of group $i$};
\node[bnote,anchor=west] at (7.45,-2.45) {the $r_j$ connectors\\of group $j$};
\node[bnote,anchor=west] at (7.45,-4.90) {the connectors of\\the other three groups};
\end{tikzpicture}
\caption{The tree decomposition of Lemma~\ref{lem:decomposition}:
a path of three main bags with the $17$ connector leaves hung off it.
Here $i$ minimizes $h(U_i)$ and $j=2$ if $i=1$, else $j=1$.
Every leaf is $N[p]=\{p\}\cup A_k\cup U_k$ for a connector
$p=p_{k,l}$; its own incident edges cover it, so
$h(N[p])\le\deg(p)\le5$.
Along the path central vertices are only added and anchors only removed,
which is what makes all vertex occurrences connected; $C_1$ carries all
anchor edges and $C_3$ all central edges.
Only the bound for $C_1$ can exceed $6$, and it does so by
$h(U_i)-1$.}
\label{fig:td-generic}
\end{figure}
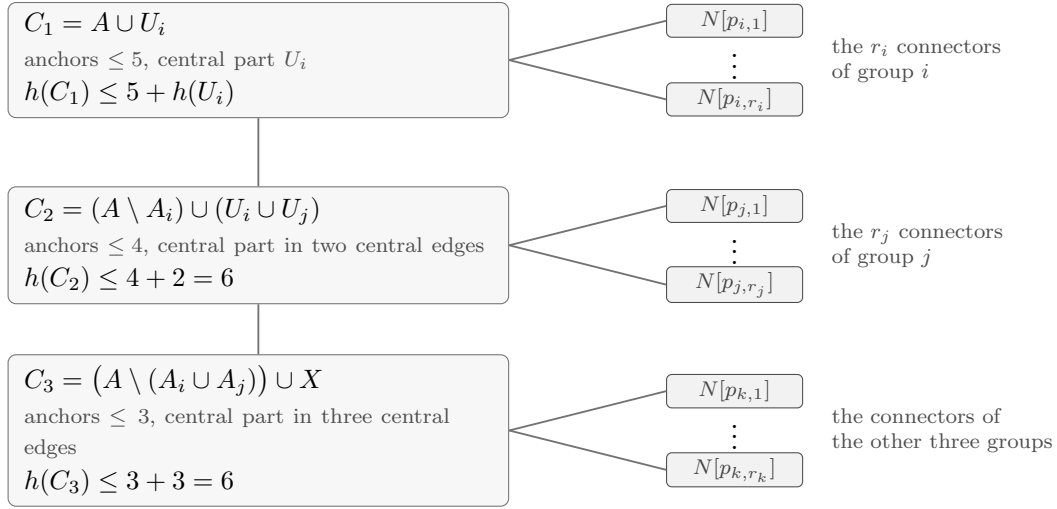

\begin{figure}[tbp]
\centering
\begin{tikzpicture}[y=1cm,x=1cm]
\node[maxbag] (C1) at (0,0)
  {$C_1=A\cup\xs{123}$ \hfill\textcolor{black!60}{\scriptsize13 vertices}\\[1pt]
   \textcolor{black!70}{\scriptsize anchors $5$; central part $U_1=\xs{123}$}\\[1pt]
   $F_E(C_1)=5+\frac{130}{101}=\frac{635}{101}$};
\node[mainbag] (C2) at (0,-2.6)
  {$C_2=(A\setminus A_1)\cup\xs{1234}$ \hfill\textcolor{black!60}{\scriptsize12 vertices}\\[1pt]
   \textcolor{black!70}{\scriptsize anchors $4$; $\xs{1234}$ covered by $\xs{13},\xs{24}$}\\[1pt]
   $F_E(C_2)=4+\frac{154}{101}=\frac{558}{101}$};
\node[mainbag] (C3) at (0,-6.3)
  {$C_3=\bigl(A\setminus(A_1\cup A_2)\bigr)\cup X$ \hfill\textcolor{black!60}{\scriptsize11 vertices}\\[1pt]
   \textcolor{black!70}{\scriptsize anchors $3$; $X$ covered by $\xs{12},\xs{35},\xs{45}$}\\[1pt]
   $F_E(C_3)=3+\frac{154}{101}=\frac{457}{101}$};
\draw[tlink] (C1) -- (C2);
\draw[tlink] (C2) -- (C3);
\def\lx{6.4}
\foreach \k/\y in {1/0.44,2/0.00,3/-0.44}
  \node[leafbag=grpI] (g_E1\k) at (\lx,\y) {$N[p_{1,\k}]$};
\foreach \k/\y in {1/-2.16,2/-2.60,3/-3.04}
  \node[leafbag=grpII] (g_E2\k) at (\lx,\y) {$N[p_{2,\k}]$};
\foreach \k/\y in {1/-4.28,2/-4.72,3/-5.16}
  \node[leafbag=grpIII] (g_E3\k) at (\lx,\y) {$N[p_{3,\k}]$};
\foreach \k/\y in {1/-5.72,2/-6.16,3/-6.60,4/-7.04}
  \node[leafbag=grpIV] (g_E4\k) at (\lx,\y) {$N[p_{4,\k}]$};
\foreach \k/\y in {1/-7.60,2/-8.04,3/-8.48,4/-8.92}
  \node[leafbag=grpV] (g_E5\k) at (\lx,\y) {$N[p_{5,\k}]$};
\foreach \t in {g_E11,g_E12,g_E13} \draw[tlink] (C1.east) -- (\t.west);
\foreach \t in {g_E21,g_E22,g_E23} \draw[tlink] (C2.east) -- (\t.west);
\foreach \t in {g_E31,g_E32,g_E33,g_E41,g_E42,g_E43,g_E44,g_E51,g_E52,g_E53,g_E54}
  \draw[tlink] (C3.east) -- (\t.west);
\node[bnote,anchor=north west,text width=110mm] at (-3.1,-9.7)
 {Leaf bags:
  $N[p_{1,k}]=\{p_{1,k}\}\cup A_1\cup\xs{123}$,
  $N[p_{2,k}]=\{p_{2,k}\}\cup A_2\cup\xs{124}$,
  $N[p_{3,k}]=\{p_{3,k}\}\cup A_3\cup\xs{34}$,
  $N[p_{4,k}]=\{p_{4,k}\}\cup A_4\cup\xs{15}$,
  $N[p_{5,k}]=\{p_{5,k}\}\cup A_5\cup\xs{25}$.
  Every leaf has $F_E=\tfrac{231}{101}$ or $\tfrac{223+10a}{101}$,
  both below $\tfrac{7}{3}$.};
\end{tikzpicture}
\caption{The complete tree decomposition from
  Figure~\ref{fig:td-generic}, assuming $i=1$ (hence $j=2$), with the
  entropic witness $F_E$ of \eqref{eq:witness} evaluated on all $20$
  bags.  Leaf colors are the anchor-group colors of
  Figure~\ref{fig:graph32}.  The maximum is attained at the single bag
  $C_1$, giving $\hwd_{F_E}(G)=635/101$; every other bag is below $6$.
  For $i\in\{4,5\}$ the shape is the same, with $U_i$ as the central
  part of $C_1$, central part $\xs{1235}$ in $C_2$ covered by
  $\xs{12},\xs{35}$ instead of $\xs{13},\xs{24}$, and the leaf fans
  redistributed among the three hosts.}
\label{fig:td-concrete}
\end{figure}

\subsection{The Bramble Lemma}

For the lower bounds, we need to show that every tree decomposition has
a single bag that captures both a central contribution and a contribution
from the anchors and connectors. More precisely, the bag should contain
some $U_i$ and meet every anchor--connector pair $K_{a,k}$. We enforce
these requirements using a \emph{bramble}: a family of connected vertex
sets that pairwise intersect or are joined by an edge. We say that two
such sets \emph{touch}. Every tree decomposition has a bag meeting all
members of a bramble (see Bellenbaum and Diestel~\cite[proof of Theorem~5]{bellenbaum2002two}).
We recall the short argument below. Our family consists
of the anchor--connector pairs together with the complements $X\setminus e$
of the central edges. Meeting all these complements forces the central
part of the bag to contain some $U_i$.

\begin{lemma}\label{lem:bramble}
Every tree decomposition of $G$ has a bag $B$ such that $B\cap X$
contains some $U_i$ and $B$ meets every $K_{a,k}$.
\end{lemma}
\begin{proof}
Consider the family
\[
\mathcal B=\{X\setminus e:e\in E_0\}
\cup\{K_{a,k}:a\in A_i,\ 1\le i\le5,\ 1\le k\le r_i\}.
\]
We first verify that $\mathcal B$ is a bramble. Each central complement
induces a path or a triangle, and each $K_{a,k}$ is an edge, so every
member is connected. There are three types of pairs to consider.
Two central complements intersect because each contains three of the
five central vertices. Next, consider two anchor--connector pairs
$K_{a,k}$ and $K_{a',\ell}$. If $a=a'$, they intersect at their common
anchor. If the anchors belong to different groups, the edge $\{a,a'\}$
joins the two pairs. If the anchors are distinct but belong to the same
group $A_i$, the edge $\{a,p_{i,\ell}\}$ joins them.
Finally, consider a central complement $X\setminus e$ and a pair
$K_{a,k}$ with $a\in A_i$. No $U_i$ is contained in a central edge,
as noted in Section~\ref{sec:graph}, so there is a vertex
$x\in U_i\setminus e$. The edge $\{x,p_{i,k}\}$ joins the complement
to the pair. Thus all members of $\mathcal B$ pairwise touch.

Now fix any tree decomposition $(T,(B_t)_{t\in V(T)})$ of $G$.
For each $S\in\mathcal B$, let
\[
T_S=\{t\in V(T):B_t\cap S\ne\varnothing\}.
\]
This is a nonempty connected subtree of $T$ as the bags containing each
vertex $v\in S$ form a nonempty connected subtree, and the subtrees
for adjacent vertices of $S$ intersect in a bag containing their edge.
Since $S$ is connected, the union of these vertex subtrees is connected.
Moreover, if two members $S,S'$ of $\mathcal B$ intersect at a vertex,
any bag containing that vertex belongs to both $T_S$ and $T_{S'}$.
If instead an edge joins $S$ to $S'$, a bag containing that edge belongs
to both subtrees. Hence the subtrees $T_S$ pairwise intersect.
The Helly property for subtrees of a tree states that a finite family
of pairwise intersecting subtrees has a common node. Applying it here
gives a node $t$ whose bag $B=B_t$ meets every member of $\mathcal B$.

In particular, for every $e\in E_0$ we have
\[
B\cap(X\setminus e)\ne\varnothing
\quad\Longleftrightarrow\quad
B\cap X\not\subseteq e.
\]
Thus the central part $B\cap X$ is contained in no central edge.
The sets $U_1,\ldots,U_5$ are exactly the minimal subsets of $X$ with
this property, as shown in Section~\ref{sec:graph}, so $B\cap X$
contains some $U_i$. The bag $B$ also meets every $K_{a,k}$, since
these pairs are members of $\mathcal B$.
\end{proof}

\subsection{The Exact Adaptive Width}

\begin{lemma}\label{lem:exact-adw}
The $32$-vertex graph $G$ satisfies $\adw(G)=6$.
\end{lemma}
\begin{proof}
For the upper bound, let $h$ be an edge dominated modular polymatroid
on $G$.  Its restriction to the nodes $X$ of the central graph $G_0$
is again modular and edge dominated.  Lemma~\ref{lem:modular-core}
gives $\min_i h(U_i)\le1$, while Lemma~\ref{lem:decomposition} implies
\[
\hwd_h(G)\le\max\{6,5+1\}=6.
\]
Taking the supremum proves $\adw(G)\le6$.

For the lower bound, let $h(B)=|B|/2$ for every $B\subseteq V(G)$.
This is a modular polymatroid, and it is edge dominated since every
edge has value $1$.  Lemma~\ref{lem:bramble} gives a bag $B$ whose
central part contains some $U_i$ and which meets every
$K_{a,k}$.  Thus $|B\cap X|\ge2$.  For each anchor
group $A_i$, either $B$ contains both anchors, or an omitted anchor
forces $B$ to contain all $r_i\ge3$ connectors of that group.  Hence
$B$ contains at least two vertices from each of the five groups of
anchors and connectors.  Consequently,
\[
|B|\ge2+5\cdot2=12,\qquad h(B)\ge6.
\]
Therefore $\adw(G)\ge\hwd_h(G)\ge6$.  This proves the lemma.
\end{proof}

\subsection{Extending a Central Witness to the Big Graph}

Given any function $g_0:2^X\to\Rp$ on the central graph, extend it
to all vertices of $G$ by
\begin{equation}\label{eq:witness}
F_{g_0}(B)=g_0(B\cap X)+
\sum_{i=1}^5\frac1{2r_i}\sum_{a\in A_i}\sum_{k=1}^{r_i}
\mathbf1_{\{B\cap K_{a,k}\ne\varnothing\}}.
\end{equation}

We notice that the sum we added to $g_0$ is a normal polymatroid,
because it is a positive linear combination of step functions
(Section~\ref{sec:basic}).  This implies that, if $g_0$ is a normal, or
linear, or almost entropic, or general polymatroid, then so is $F_{g_0}$.
Each of the four lower bounds below is obtained by substituting for
$g_0$ the corresponding witness on the central graph from
Section~\ref{sec:core-bounds}.  We record conditions for \eqref{eq:witness} to be edge dominated
and a lower bound on its value on a bag of every tree decomposition.  Note that
$F_{g_0}(B)=g_0(B)$ for $B\subseteq X$, since every $K_{a,k}$ is
disjoint from $X$.

\begin{lemma}\label{lem:feasible}
Suppose that $g_0(e)\le1$ for every $e\in E_0$, that $g_0(X_i)\le2/3$
for $i=1,\ldots,4$, and that $g_0(X_5)\le3/4$.  Then $F_{g_0}$ is edge
dominated on $G$.
\end{lemma}
\begin{proof}
Central edges have value $g_0(e)\le1$.  Each anchor meets
all $r_i$ coverage sets belonging to it, for total value $1/2$;
an anchor--anchor edge therefore has value one. A connector in group
$i$ meets two coverage sets, one for each anchor, for value $1/r_i$.
Consequently an anchor--connector edge has value
$1/2+1/(2r_i)\le1$.

A central--connector edge $\{x,p_{i,k}\}$ has value $g_0(x)+1/r_i$,
where $x\in U_i$.  For $i=1,2,3$ we have $r_i=3$ and
$U_i\subseteq\xs{1234}$, so this is at most $2/3+1/3=1$; for $i=4,5$ we
have $r_i=4$ and $U_i\subseteq\xs{125}$, so it is at most
$3/4+1/4=1$.  These are all edge types.
\end{proof}

\begin{lemma}\label{lem:witness-bound}
For every monotone $g_0:2^X\to\Rp$, every tree decomposition of $G$
has a bag $B$ such that
\[
F_{g_0}(B)\ge\min_i g_0(U_i)+5.
\]
\end{lemma}
\begin{proof}
Let $B$ be a bag given by Lemma~\ref{lem:bramble}.
Meeting every $K_{a,k}$ makes every coverage term in
\eqref{eq:witness} attain its full value. There are ten anchors, each
contributing $1/2$, for a total of five.  Since $g_0$ is monotone,
$g_0(B\cap X)\ge g_0(U_i)\ge\min_i g_0(U_i)$, and the displayed bound
follows.
\end{proof}

\subsection{The Exact Normal Width}

\begin{lemma}\label{lem:exact-normw}
The $32$-vertex graph $G$ satisfies $\normw(G)=69/11$.
\end{lemma}
\begin{proof}
  For the upper bound, let $h$ be an edge-dominated normal polymatroid
  on $G$.  Its restriction to the nodes $X$ of the central graph $G_0$
  is again normal, since deleting the variables outside $X$ from a
  coverage function leaves a coverage function, and it is edge
  dominated on $(X,E_0)$.  So Theorem~\ref{thm:central} (more
  precisely, Lemma~\ref{lem:normal-core}) gives
  $\min_i h(U_i)\le14/11$, while Lemma~\ref{lem:decomposition} implies
\[
\hwd_h(G)\le\max\left\{6,\ 5+\frac{14}{11}\right\}=\frac{69}{11},
\]
since $69/11>6$.  Taking the supremum proves $\normw(G)\le69/11$.

For the lower bound, let $F_N:=F_{g_N}$ be the extension
\eqref{eq:witness} of the normal witness $g_N$ of
\eqref{eq:normal-witness}.
This is normal.  Indeed, $g_N$ is a nonnegative combination of step
functions on $X$ by \eqref{eq:normal-witness}, and for $K\subseteq X$
the extension $B\mapsto h^K(B\cap X)$ is again the step function $h^K$,
now on $V(G)$; each term
$\mathbf1_{\{B\cap K_{a,k}\ne\varnothing\}}$ is the step function
$h^{K_{a,k}}$; and a nonnegative combination of step functions is
normal.

$F_N$ is edge dominated by Lemma~\ref{lem:feasible}: central edges have
value $g_N(e)\le1$ by Lemma~\ref{lem:normal-tight}, and by
\eqref{eq:normal-singletons} the singletons satisfy $g_N(X_i)\le7/11<2/3$
for $i=1,\ldots,4$ and $g_N(X_5)=8/11<3/4$.

Lemma~\ref{lem:witness-bound} now gives a bag $B$ with
\[
F_N(B)\ \ge\ \min_i g_N(U_i)+5=\frac{14}{11}+5=\frac{69}{11},
\]
and $\normw(G)\ge\hwd_{F_N}(G)\ge69/11$.
\end{proof}

Comparing with Lemma~\ref{lem:exact-adw},
\[
\normw(G)-\adw(G)=\frac{69}{11}-6=\frac3{11}>0,
\]
so $G$ separates adaptive width from normal width.
The normal witness $g_N$ also violates the modular identity
\eqref{eq:modular-core}: its left-hand side is
$
g_N(U_3)+g_N(U_4)+g_N(U_5)=\frac{42}{11},
$
while its right-hand side is
$
g_N(\xs{12})+g_N(\xs{35})+g_N(\xs{45})
=\frac{32}{11}.
$

\subsection{The Exact Linear Width}
\label{sec:linear-upper}

\begin{lemma}\label{lem:exact-linw}
The $32$-vertex graph $G$ satisfies $\linw(G)=44/7$.
\end{lemma}
\begin{proof}
  For the upper bound, let $h\in\text{Lin}_{32}$ be edge
  dominated on $G$. Its restriction to the nodes $X$ of the central
  graph $G_0$ belongs to $\text{Lin}_X$ and is edge dominated
  on $(X,E_0)$. Lemma~\ref{lem:ingleton} gives $\min_i h(U_i)\le9/7$.
  Lemma~\ref{lem:decomposition} implies
\[
\hwd_h(G)\le\max\left\{6,\ 5+\frac97\right\}=\frac{44}{7},
\]
since $44/7>6$.  Taking the supremum proves $\linw(G)\le44/7$.

For the lower bound, let $F_L:=F_{g_L}$ be the
extension~\eqref{eq:witness} of the linear witness $g_L=r/14$ of
Lemma~\ref{lem:ingleton-tight}.  This is linear.  Indeed $g_L$ is a
linear polymatroid on $X$, and extending it by $B\mapsto g_L(B\cap X)$
keeps it one, assigning the empty set of vectors to every vertex
outside $X$; each coverage term is normal, hence linear; and clearing
denominators,
\[
168\,F_L=12\,r+28\sum_{i\le3}\sum_{a\in A_i}\sum_k
\mathbf1_{\{\,\cdot\,\cap K_{a,k}\ne\varnothing\}}
+21\sum_{i\ge4}\sum_{a\in A_i}\sum_k
\mathbf1_{\{\,\cdot\,\cap K_{a,k}\ne\varnothing\}}
\]
is a nonnegative \emph{integer} combination of rank functions, hence
itself the rank function of the direct sum of the corresponding vector
configurations.  So $F_L=\frac1{168}r'$ with $r'$ a rank function, and
$F_L\in\text{Lin}_{32}$.

$F_L$ is edge dominated by Lemma~\ref{lem:feasible}: central edges have
value $g_L(e)\le1$ by Lemma~\ref{lem:ingleton-tight}, and by
\eqref{eq:linear-singletons} the singletons satisfy $g_L(X_i)\le9/14<2/3$
for $i=1,\ldots,4$ and $g_L(X_5)=5/7<3/4$.

Lemma~\ref{lem:witness-bound} now gives a bag $B$ with
\[
F_L(B)\ \ge\ \min_i g_L(U_i)+5=\frac97+5=\frac{44}{7},
\]
giving $\linw(G)\ge\hwd_{F_L}(G)\ge44/7$.
\end{proof}

Comparing with Lemma~\ref{lem:exact-normw},
\[
\linw(G)-\normw(G)=\frac{44}{7}-\frac{69}{11}=\frac{484-483}{77}
=\frac1{77}>0,
\]
so $G$ separates normal width from linear width.  The linear witness
also violates the covering inequality \eqref{eq:normal-core}: the
configuration \eqref{eq:linear-witness} has
$r(\xs{123})+r(\xs{124})+5r(\xs{34})+2r(\xs{15})+2r(\xs{25})=198$ while
$2r(\xs{13})+2r(\xs{14})+2r(\xs{23})+2r(\xs{24})+3r(\xs{35})+3r(\xs{45})=196$.

\subsection{The Entropic Width}

\begin{lemma}\label{lem:entw-bounds}
The $32$-vertex graph $G$ satisfies
\[
\frac{635}{101}\le\entw(G)\le\frac{183}{29}.
\]
\end{lemma}
\begin{proof}
For the upper bound, the restriction to $X$ of any edge dominated
almost entropic function on $G$ is almost entropic on $X$ and edge
dominated on $(X,E_0)$, so Lemma~\ref{lem:zy-core} applies.
Lemma~\ref{lem:decomposition} then gives
\[
\hwd_h(G)\le\max\left\{6,\ 5+\frac{38}{29}\right\}
=\frac{183}{29}.
\]
Therefore $\entw(G)\le183/29$.

For the lower bound, let $g_E$ be the almost entropic polymatroid on the central graph
constructed in Lemma~\ref{lem:zy-tight}, and let $F_E:=F_{g_E}$ be its
extension \eqref{eq:witness} to $G$.
Every coverage term is the entropy profile of a random variable
copied to the two vertices in $K_{a,k}$, with constant values at all
other vertices. Choose these variables independently of one another
and of the central graph. Thus $F_E$ is almost entropic.
There is also a direct finite realization after multiplying by $2424$:
take twelve independent copies of the vector in
\eqref{eq:finite-entropy}, and for each $K_{a,k}$ share
$2424/(2r_i)$ fresh independent fair bits between its two vertices.
These numbers are $404$ for $r_i=3$ and $303$ for $r_i=4$.
The resulting entropy vector is exactly $2424F_E$.

By \eqref{eq:core-edges} and \eqref{eq:core-singletons} the function
$g_E$ satisfies the hypotheses of Lemma~\ref{lem:feasible}, so $F_E$ is
edge dominated; and $\min_i g_E(U_i)=130/101$ by \eqref{eq:core-large},
so Lemma~\ref{lem:witness-bound} produces a bag $B$ with
$F_E(B)\ge130/101+5=635/101$.  Hence
$\entw(G)\ge\hwd_{F_E}(G)\ge635/101$. In fact
Lemma~\ref{lem:decomposition} shows $\hwd_{F_E}(G)=635/101$, because
$\min_i F_E(U_i)=130/101$ and $5+130/101=635/101>6$.
\end{proof}

Comparing with Lemma~\ref{lem:exact-linw},
\[
\entw(G)-\linw(G)\ge
\frac{635}{101}-\frac{44}{7}
=\frac1{707}>0,
\]
so $G$ separates linear width from entropic width.
The almost entropic witness $g_E$ also violates the linear rank
inequality \eqref{eq:ing-core}. Indeed, its left-hand side is
\[
g_E(\xs{123})+g_E(\xs{124})+3g_E(\xs{34})
+g_E(\xs{15})+g_E(\xs{25})
=\frac{894+20a}{101}>9,
\]
where the strict inequality follows from $a>4/5$.
Its right-hand side is $9$, since $g_E(e)=1$ for every central
edge.

\subsection{The Exact Submodular Width}

\begin{lemma}\label{lem:exact-subw}
The $32$-vertex graph $G$ satisfies $\subw(G)=19/3$.
\end{lemma}
\begin{proof}
  For the upper bound, Lemma~\ref{lem:subw-core} and
  Lemma~\ref{lem:decomposition} give
  $\hwd_h(G)\le\max\{6,5+4/3\}=19/3$ for every edge-dominated
  polymatroid $h$, proving $\subw(G)\le19/3$.

For the lower bound, let $g_P$ be the edge dominated polymatroid
of Lemma~\ref{lem:subw-tight}: it satisfies
$g_P(e)=1$ for $e\in E_0$, and $g_P(B)=4/3$ whenever $B$ is contained
in no central edge, so $g_P(U_i)=4/3$ for every $i$.

Let $F_P:=F_{g_P}$ be the extension \eqref{eq:witness} of $g_P$.
This is a polymatroid, being a nonnegative combination of $g_P$ and
the coverage functions $B\mapsto\mathbf1_{\{B\cap K\ne\varnothing\}}$,
and it is edge dominated by Lemma~\ref{lem:feasible}, since
$g_P(e)=1$ for $e\in E_0$ and $g_P(X_i)=2/3$ for every $i$.

Lemma~\ref{lem:witness-bound} now gives a bag $B$ with
\[
F_P(B)\ \ge\ \min_i g_P(U_i)+5=\frac43+5=\frac{19}{3}.
\]
This proves the matching lower bound.
\end{proof}

Comparing with Lemma~\ref{lem:entw-bounds},
\[
\subw(G)-\entw(G)\ge
\frac{19}{3}-\frac{183}{29}=\frac2{87}>0,
\]
so $G$ separates entropic width from submodular width.
The polymatroid witness $g_P$ also violates
inequality \eqref{eq:zy-core}. Its left-hand side is
$
5g_P(\xs{123})+5g_P(\xs{124})+11g_P(\xs{34})
+4g_P(\xs{15})+4g_P(\xs{25})
=\frac{116}{3}>38.
$
Its right-hand side is $38$, since $g_P(e)=1$ for every central
edge.

\subsection{Proof of the Separation Theorem}

\begin{proof}[Proof of Theorem~\ref{thm:main}]
Let $G$ be the $32$-vertex graph defined in
Section~\ref{sec:graph}.
Lemmas~\ref{lem:exact-adw}, \ref{lem:exact-normw},
\ref{lem:exact-linw}, \ref{lem:entw-bounds}, and
\ref{lem:exact-subw} give, respectively,
\[
\begin{gathered}
\adw(G)=6,\qquad
\normw(G)=\frac{69}{11},\qquad
\linw(G)=\frac{44}{7},\\
\frac{635}{101}\le\entw(G)\le\frac{183}{29},\qquad
\subw(G)=\frac{19}{3}.
\end{gathered}
\]
The strict separations follow from
\[
\begin{aligned}
\frac{69}{11}-6&=\frac3{11}>0,&
\frac{44}{7}-\frac{69}{11}&=\frac1{77}>0,\\
\frac{635}{101}-\frac{44}{7}&=\frac1{707}>0,&
\frac{19}{3}-\frac{183}{29}&=\frac2{87}>0.
\end{aligned}
\]
This proves the theorem.
\end{proof}

\paragraph{Computer-aided Validation.}
This proves the separation for the $32$-vertex graph specified here.
The proof is algebraic and combinatorial; the companion script
\texttt{verify.py} independently checks the entropy
table, edge feasibility, every displayed decomposition, the
bramble incidences, every bag value drawn in
Figure~\ref{fig:td-concrete}, all three information-inequality certificates
\eqref{eq:ing-cert}, \eqref{eq:zy-cert} and \eqref{eq:shannon-core},
the equivalence of \eqref{eq:ing}--\eqref{eq:zy} with their usual
formulations, the atom-by-atom check behind \eqref{eq:normal-core}
together with the normal witness \eqref{eq:normal-witness}, every rank
of the subspace arrangement \eqref{eq:linear-witness}, and the exact
normal-, linear- and submodular-width witnesses.  It also parses the
Ti\emph{k}Z source of Figure~\ref{fig:graph32} and checks that the
edge set actually drawn there is exactly $E(G)$, and that the pentagon
claim made in its caption holds.  Finally, it confirms by linear
programming that the bounds $14/11$, $9/7$ and $4/3$ cannot be
improved: they are the exact optima of $\max_h\min_i h(U_i)$ over the
edge dominated normal polymatroids on $G_0$, over the edge dominated
polymatroids satisfying every instance of \eqref{eq:ing}, and over all
edge dominated polymatroids, respectively, and the witnesses $g_N$,
$g_L$ and $g_P$ attain them.  It also confirms that $38/29$ cannot be
improved over the Zhang--Yeung relaxation; as discussed after
Lemma~\ref{lem:zy-tight}, that leaves the entropic value
undetermined.

\section{Conclusions}
\label{sec:conclusions}

We have separated the notions of adaptive, normal, linear, entropic, and
submodular width by using a single graph $G$ with 32 vertices.  We leave
two open problems.  The first is to study the connection between the
various notions of width and fine-grained complexity lower bounds.
The currently known fine-grained complexity lower bounds
in~\cite{DBLP:conf/icalp/FanKZ23} use a \emph{clique embedding}, which
is a further restriction of a normal polymatroid.  For patterns with
higher complexity, Bringmann and
Gorbachev~\cite{DBLP:conf/stoc/BringmannG25} use a reduction from the
3SUM conjecture, which appears to have similarities with the
linear width.  The question we raise is if there exists a general rule
connecting the notions of width with fine-grained complexity
reductions.

Second, we leave open the question of identifying ``natural'' classes
of graphs where the submodular width coincides with simpler notions of
width.  If matched with appropriate fine-grained reductions, these
restricted classes are the most likely candidates for patterns for which
the exact complexity can be characterized, i.e. the lower and upper
bounds coincide.

\section*{Acknowledgements}

The work of Lanzinger and Merkl was partially supported by the Vienna Science and Technology Fund (WWTF) [10.47379/ICT2201].
Suciu was partially supported by NSF 2507117 and NSF SHF 2312195.

\bibliographystyle{alpha}
\bibliography{bib}

\end{document}